\documentclass[3p,cmr,11pt,letter]{elsarticle}
\usepackage{amssymb}
\usepackage{amsmath}
\usepackage{amsthm}

\DeclareMathAlphabet{\mathcal}{OMS}{cmsy}{m}{n}

\makeatletter
\def\ps@pprintTitle{%
 \let\@oddhead\@empty
 \let\@evenhead\@empty
 \def\@oddfoot{\centerline{\thepage}}%
 \let\@evenfoot\@oddfoot}
\makeatother

\newcommand{\bbC}{\mathbb{C}}
\newcommand{\bbF}{\mathbb{F}}
\newcommand{\bbR}{\mathbb{R}}
\newcommand{\bbZ}{\mathbb{Z}}

\newcommand{\bfA}{\mathbf{A}}
\newcommand{\bfB}{\mathbf{B}}
\newcommand{\bfD}{\mathbf{D}}
\newcommand{\bfG}{\mathbf{G}}
\newcommand{\bfH}{\mathbf{H}}
\newcommand{\bfI}{\mathbf{I}}
\newcommand{\bfJ}{\mathbf{J}}
\newcommand{\bfS}{\mathbf{S}}
\newcommand{\bfU}{\mathbf{U}}
\newcommand{\bfy}{\mathbf{y}}
\newcommand{\bfz}{\mathbf{z}}

\newcommand{\bfzero}{\boldsymbol{0}}
\newcommand{\bfone}{\boldsymbol{1}}
\newcommand{\bfphi}{\boldsymbol{\varphi}}
\newcommand{\bfPhi}{\boldsymbol{\Phi}}
\newcommand{\bfPi}{\boldsymbol{\Pi}}

\newcommand{\calB}{\mathcal{B}}
\newcommand{\calD}{\mathcal{D}}
\newcommand{\calG}{\mathcal{G}}
\newcommand{\calV}{\mathcal{V}}

\newcommand{\rmC}{\mathrm{C}}
\newcommand{\rmi}{\mathrm{i}}
\newcommand{\rmN}{\mathrm{N}}

\newcommand{\Tr}{\operatorname{Tr}}
\newcommand{\SRG}{\operatorname{SRG}}
\newcommand{\BIBD}{\operatorname{BIBD}}
\newcommand{\RBIBD}{\operatorname{RBIBD}}

\newcommand{\abs}[1]{|{#1}|}
\newcommand{\bigabs}[1]{\bigl|{#1}\bigr|}

\newcommand{\bigbracket}[1]{\bigl[{#1}\bigr]}

\newcommand{\set}[1]{\{{#1}\}}
\newcommand{\bigset}[1]{\bigl\{{#1}\bigr\}}
\newcommand{\Bigset}[1]{\Bigl\{{#1}\Bigr\}}
\newcommand{\biggset}[1]{\biggl\{{#1}\biggr\}}

\newcommand{\norm}[1]{\|{#1}\|}

\newcommand{\ip}[2]{\langle{#1},{#2}\rangle}
\newcommand{\biggip}[2]{\biggl\langle{#1},{#2}\biggr\rangle}

\newcommand{\alphi}{\renewcommand{\labelenumi}{(\alph{enumi})}}
\newcommand{\romani}{\renewcommand{\labelenumi}{(\roman{enumi})}}
\newcommand{\romanii}{\renewcommand{\labelenumii}{(\roman{enumii})}}

\newtheorem{theorem}{Theorem}[section]
\newtheorem{corollary}[theorem]{Corollary}
\newtheorem{lemma}[theorem]{Lemma}

\theoremstyle{definition}
\newtheorem{definition}[theorem]{Definition}

\begin{document}
\begin{frontmatter}
\title{Equiangular tight frames with centroidal symmetry}

\author[AFIT]{Matthew Fickus}
\ead{Matthew.Fickus@gmail.com}
\author[Cincy]{John Jasper}
\author[AFIT]{Dustin G.\ Mixon}
\author[AFIT]{Jesse D.\ Peterson}
\author[AFIT]{Cody E.\ Watson}

\address[AFIT]{Department of Mathematics and Statistics, Air Force Institute of Technology, Wright-Patterson AFB, OH 45433}
\address[Cincy]{Department of Mathematical Sciences, University of Cincinnati, Cincinnati, OH 45221}


\begin{abstract}
An equiangular tight frame (ETF) is a set of unit vectors whose coherence achieves the Welch bound, and so is as incoherent as possible.
Though they arise in many applications, only a few methods for constructing them are known.
Motivated by the connection between real ETFs and graph theory, we introduce the notion of ETFs that are symmetric about their centroid.
We then discuss how well-known constructions, such as harmonic ETFs and Steiner ETFs, can have centroidal symmetry.
Finally, we establish a new equivalence between centroid-symmetric real ETFs and certain types of strongly regular graphs (SRGs).
Together, these results give the first proof of the existence of certain SRGs, as well as the disproofs of the existence of others.
\end{abstract}

\end{frontmatter}
\section{Introduction}
Let $m\leq n$ be positive integers and let $\set{\bfphi_i}_{i=1}^{n}$ be a sequence of unit vectors in $\bbF^m$ where the field $\bbF$ is either the real line $\bbR$ or the complex plane $\bbC$.
The quantity $\max_{i\neq j}\abs{\ip{\bfphi_i}{\bfphi_j}}$ is known as the \textit{coherence} of $\set{\bfphi_i}_{i=1}^{n}$.
In certain real-world applications, one seeks a sequence of $n$ unit norm vectors in $\bbF^m$ whose coherence is as small as possible.
Geometrically speaking, this is equivalent to packing lines in Euclidean space: for any real unit vectors $\bfphi_i$ and $\bfphi_j$, we have $\abs{\ip{\bfphi_i}{\bfphi_j}}=\cos(\theta_{i,j})$ where $\theta_{i,j}$ is the interior angle of the lines spanned by $\bfphi_i$ and $\bfphi_j$;
finding unit vectors $\set{\bfphi_i}_{i=1}^{n}$ with minimal coherence is thus equivalent to arranging $n$ lines so that the minimum pairwise angle between any two lines is as large as possible.
The \textit{Welch bound} is the most famous example of an explicit lower bound on the coherence~\cite{Welch74,StrohmerH03}:
\begin{theorem}
\label{theorem.Welch bound}
Given positive integers $m\leq n$, for any unit norm vectors $\set{\bfphi_i}_{i=1}^{n}$ in $\bbF^m$ we have
\begin{equation}
\label{equation.Welch bound}
\max_{i\neq j}\abs{\ip{\bfphi_i}{\bfphi_j}}
\geq
\sqrt{\tfrac{n-m}{m(n-1)}},
\end{equation}
where equality holds if and only if $\set{\bfphi_i}_{i=1}^{n}$ is an \textit{equiangular tight frame} (ETF) for $\bbF^m$.
\end{theorem}

\setlength{\arraycolsep}{2pt}

To fully understand this fact, we first establish some notation and terminology.
For any vectors $\set{\bfphi_i}_{i=1}^{n}$ in $\bbF^m$, the corresponding \textit{synthesis operator} is the $m\times n$ matrix $\bfPhi$ which has the vectors $\set{\bfphi_i}_{i=1}^{n}$ as its columns, namely the operator $\bfPhi:\bbF^n\rightarrow\bbF^m$, $\bfPhi\bfy=\sum_{i=1}^{n}\bfy(i)\bfphi_i$.
Composing $\bfPhi$ with its $n\times m$ adjoint (conjugate transpose) $\bfPhi^*$ yields the $m\times m$ \textit{frame operator} $\bfPhi\bfPhi^*$
as well as the $n\times n$ \textit{Gram matrix} $\bfPhi^*\bfPhi$ whose $(i,j)$th entry is $(\bfPhi^*\bfPhi)(i,j)=\ip{\bfphi_i}{\bfphi_j}$.
We say $\set{\bfphi_i}_{i=1}^{n}$ is a \textit{tight frame} if $\bfPhi$ is perfectly conditioned, that is, if there exists $\alpha>0$ such that $\bfPhi\bfPhi^*=\alpha\bfI$.
This is equivalent to having the rows of $\bfPhi$ be orthogonal and equal norm.
We say $\set{\bfphi_i}_{i=1}^{n}$ is \textit{equiangular} when each $\bfphi_i$ has unit norm and the value of $\abs{\ip{\bfphi_i}{\bfphi_j}}$ is constant over all choices of $i\neq j$,
namely when the diagonal entries of $\bfPhi^*\bfPhi$ are $1$ while the off-diagonal entries have constant modulus.

An ETF is a tight frame whose vectors are equiangular.
This means the synthesis operator $\bfPhi$ has equal-norm orthogonal rows and unit-norm equiangular columns.
Note that in this case the tight frame constant $\alpha$ is necessarily the frames \textit{redundancy} $\frac nm$ since $m\alpha=\Tr(\alpha\bfI)=\Tr(\bfPhi\bfPhi^*)=\Tr(\bfPhi^*\bfPhi)=\sum_{i=1}^{n}\norm{\bfphi_i}^2=n$.
Theorem~\ref{theorem.Welch bound} states that Welch's lower bound on coherence is achieved precisely when the columns of $\bfPhi$ form an ETF for $\bbF^m$;
see~\cite{JasperMF14} for a short proof of this result.

Because of their minimal coherence, ETFs are useful in a number of real-world applications, including waveform design for wireless communication~\cite{StrohmerH03}, compressed sensing~\cite{BajwaCM12,BandeiraFMW13} and algebraic coding theory~\cite{JasperMF14}.
They are special instances of \textit{equi-chordal} and \textit{equi-isoclinic tight fusion frames}, topics which have garnered interest in recent years~\cite{KutyniokPCL09,BachocE13}.
In spite of these facts, only a few methods for constructing ETFs are known.
Real ETFs in particular are equivalent to a certain class of very symmetric graphs known as \textit{strongly regular graphs} (SRGs).
Much of the work behind this equivalence was pioneered by J.~J.~Seidel and his contemporaries~\cite{CorneilM91};
see~\cite{Waldron09} for a modern discussion of these ideas.
For the frame community, this equivalence is invaluable since it allows us to leverage the rich SRG literature,
notably the SRG existence tables in a book chapter~\cite{Brouwer07} and website~\cite{Brouwer15} by Brouwer.

In this paper, we discuss special classes of ETFs that possess certain types of symmetry.
Specifically, we consider ETFs with the property that the all-ones vector $\bfone$ is an eigenvector of the Gram matrix $\bfPhi^*\bfPhi$.
In the next section, we motivate this concept and review several well-known facts about ETFs that we will use later on.
In Section~3, we prove that having $\bfone$ be an eigenvector of $\bfPhi^*\bfPhi$ is equivalent to the ETF being symmetric about its centroid in one of two possible ways.
We then discuss the degree to which the \textit{harmonic ETFs} of~\cite{StrohmerH03,XiaZG05,DingF07} and the~\textit{Steiner ETFs} of~\cite{FickusMT12,JasperMF14} have such centroidal symmetries.
In the fourth section, we prove that real ETFs with centroidal symmetry are equivalent to certain types of SRGs.
Traditionally, $m\times n$ real ETFs are equated to certain SRGs on $n-1$ vertices; we show how $m\times n$ real ETFs with centroidal symmetry can moreover be equated to certain SRGs on $n$ vertices.
This relationship was briefly explored in Section~4.4 of~\cite{Yu14},
and our proof techniques here are similar to those used in a related study of non-equiangular \textit{two-distance} tight frames~\cite{BargGOY15}.
Our new equivalence allows frame theory and combinatorial design to better inform each other.
It is an abstraction of the realization that the method for constructing $n$-vector Steiner ETFs~\cite{FickusMT12} is essentially identical to Goethals and Seidel's method for constructing $n$-vertex SRGs~\cite{GoethalsS70}.
Exploiting this new equivalence, we give the first proof of the existence of certain SRGs and the nonexistence of others.

\section{Preliminaries and Motivation}

In this section we review some basic concepts and facts about ETFs and SRGs that we will need later on.
All of the results given in this section are well known.
We begin by discussing several simple ways of modifying an existing ETF $\set{\bfphi_i}_{i=1}^{n}$ into a different one.
For instance, $\set{\bfU\bfphi_{\sigma(i)}}_{i=1}^{n}$ is an ETF for any unitary matrix $\bfU\in\bbF^{m\times m}$ and permutation $\sigma:[n]\rightarrow[n]$.
We can also \textit{sign} our ETF:  $\set{z_i\bfphi_i}_{i=1}^{n}$ is an ETF for any sequence of unimodular scalars $\set{z_i}_{i=1}^{n}$.
Finally, we can produce an $(n-m)\times n$ ETF from an $m\times n$ ETF by taking any one of its \textit{Naimark complements}:
\begin{lemma}
\label{lemma.Naimark complements}
Let $\bfPhi$ be an $m\times n$ synthesis operator of an ETF with $m<n$, let $\alpha=\frac nm$, and let $\tilde{\bfPhi}$ be any $(n-m)\times n$ matrix whose rows form an orthogonal basis for the orthogonal complement of the row space of $\bfPhi$ and have squared norm $\tilde{\alpha}=\frac{n-m}{m}$.
Then $\tilde{\bfPhi}$ is the synthesis operator of an ETF and moreover, the two Gram matrices satisfy
\begin{equation*}
\bfI
=\tfrac1{\alpha}\bfPhi^*\bfPhi+\tfrac1{\tilde{\alpha}}\tilde{\bfPhi}^*\tilde{\bfPhi}.
\end{equation*}
\end{lemma}
The standard proof of this result is given in many places; see~\cite{FickusW15} for example.
The main idea is that taking $\tilde{\bfPhi}$ in this fashion makes the rows of the matrix
\begin{equation*}
\left[\begin{array}{r}\alpha^{-\frac12}\bfPhi\\\tilde{\alpha}^{-\frac12}\tilde{\bfPhi}\end{array}\right]
\end{equation*}
into an orthonormal basis for $\bbF^m$; this matrix is thus unitary and so also has orthonormal columns, implying the Gram matrix relation and thus the equiangularity of $\tilde{\bfPhi}$.

Most known direct constructions of ETFs fall into one of two categories: constructions of synthesis operators~\cite{XiaZG05,DingF07,FickusMT12,JasperMF14}, and constructions of Gram matrices~\cite{StrohmerH03,HolmesP04,Waldron09}.
Though our work here involves both approaches, most of the motivation for it comes from the latter, which relies on the following well-known result:
\begin{lemma}
\label{lemma.Gram matrix ETF characterization}
An $n\times n$ self-adjoint matrix $\bfG$ is the Gram matrix of an ETF if and only if
\begin{enumerate}
\romani
\item $\bfG^2=\alpha\bfG$ for some $\alpha\in\bbF$,
\item $\bfG(i,i)=1$ for all $i$,
\item there exists $\beta\in\bbF$ such that $\abs{\bfG(i,j)}=\beta$ for all $i\neq j$.
\end{enumerate}
\end{lemma}
The standard proof of this result is given in~\cite{FickusW15}, for example.
In brief, the ($\Rightarrow$) direction is straightforward, realizing that if $\bfG=\bfPhi^*\bfPhi$ where $\bfPhi\bfPhi^*=\alpha\bfI$ then $\bfG^2=\bfPhi^*(\bfPhi\bfPhi^*)\bfPhi=\alpha\bfG$.
For the ($\Leftarrow$) direction, let $\bfPhi$ be $\sqrt{\alpha}$ times the adjoint of the matrix whose columns form an orthonormal basis for the eigenspace of $\bfG$ corresponding to eigenvalue $\alpha$.
Note that with this approach, the dimension $m$ of the space the ETF spans is the multiplicity of $\alpha$ as an eigenvalue of $\bfG$;
since the only eigenvalues of $\bfG$ are $0$ and $\alpha$, this is easily computed as $m=\frac{\Tr(\bfG)}{\alpha}=\frac{n}{\alpha}$.

The Gram-matrix approach for constructing ETFs is especially attractive in the real case since in that setting the off-diagonal entries of $\bfG$ can only have value $\pm\beta$.
We can thus convert $\bfG$ into the adjacency matrix $\bfA$ of a graph by changing its diagonal entries to zero while changing the off-diagonal values of $\beta$ and $-\beta$ to $1$ and $0$, respectively:
\begin{equation}
\label{equation.ETF to SRG conversion}
\bfA=\tfrac1{2\beta}\bfG-\tfrac{\beta+1}{2\beta}\bfI+\tfrac12\bfJ.
\end{equation}
Here, $\bfJ$ is an all-ones matrix.
That is, $\bfJ=\bfone\bfone^*$ where $\bfone$ is an all-ones vector.
The graphs that arise from real ETFs according to~\eqref{equation.ETF to SRG conversion} are special:
solving for $\bfG$ gives $\bfG=2\beta\bfA+(\beta+1)\bfI-\beta\bfJ$, at which point the fact that $\bfG^2=\alpha\bfG$
implies that $\bfA$ necessarily satisfies
\begin{equation}
\label{equation.ETF condition on adjacency matrix}
[2\beta\bfA+(\beta+1)\bfI-\beta\bfJ]^2=\alpha(2\beta\bfA+(\beta+1)\bfI-\beta\bfJ).
\end{equation}
Conversely, if $\bfA$ is any adjacency matrix that satisfies~\eqref{equation.ETF condition on adjacency matrix} for some scalars $\alpha$ and $\beta$,
then the matrix $\bfG$ defined as $2\beta\bfA+(\beta+1)\bfI-\beta\bfJ$ satisfies all three conditions of Lemma~\ref{lemma.Gram matrix ETF characterization} and thus yields an ETF.
As such, condition~\eqref{equation.ETF condition on adjacency matrix} provides a graph-theoretic characterization of real ETFs.

To simplify~\eqref{equation.ETF condition on adjacency matrix}, it would be nice if $\bfA$ were \textit{regular}, namely if every vertex of the corresponding graph has the same number of neighbors, denoted $k$.
This happens precisely when $\bfA\bfone=k\bfone$, namely when $\bfA\bfJ=k\bfJ=\bfJ\bfA$.
In particular, having this property would allow us to expand the square in~\eqref{equation.ETF condition on adjacency matrix}.
Unfortunately, it turns out that adjacency matrices arising from real ETFs according to~\eqref{equation.ETF to SRG conversion} are not regular, in general.
The good news is that there are two ways of bypassing this issue.
The first way is the traditional method summarized in~\cite{Waldron09}.
The second way is new, and is the subject of this paper.
Both ways involve graphs which are strongly regular.

\subsection{Strongly regular graphs and their traditional relationship with real equiangular tight frames}

In general, a $k$-regular graph on $v$ vertices is said to be \textit{strongly regular} with nonnegative integer parameters $(v,k,\lambda,\mu)$ if any two neighbors have exactly $\lambda$ neighbors in common while any two nonneighbors have exactly $\mu$ neighbors in common.
A strongly regular graph with such parameters is often denoted an $\SRG(v,k,\lambda,\mu)$.
This combinatorial definition has a simple algebraic characterization in terms of the graph's $v\times v$ adjacency matrix $\bfA$:
since $\bfA^2(i,j)$ counts the number of two-step paths from vertex $i$ to $j$, a given graph is an $\SRG(v,k,\lambda,\mu)$ if and only if
\begin{equation}
\label{equation.pre definition of srg}
\bfA^2(i,j)=\left\{\begin{array}{cl}k,&\ i=j,\\\lambda,&\ i\neq j, \bfA(i,j)=1,\\\mu,&\ i\neq j, \bfA(i,j)=0,\end{array}\right.
\end{equation}
namely if and only if
\begin{equation}
\label{equation.definition of srg}
\bfA^2
=(\lambda-\mu)\bfA+(k-\mu)\bfI+\mu\bfJ.
\end{equation}
Because of its simplicity, we take~\eqref{equation.definition of srg} as our definition of an SRG.
To show that a given adjacency matrix $\bfA$ corresponds to an SRG, note it suffices to show that there exist real numbers $x,y,z$ such that $\bfA^2=x\bfA+y\bfI+z\bfJ$.
Indeed, in this case we can define $k=y+z$, $\lambda=x+z$ and $\mu=z$ to obtain~\eqref{equation.definition of srg} which is equivalent to~\eqref{equation.pre definition of srg}; excluding $\bfA=\bfzero$ and $\bfA=\bfJ-\bfI$ (both of which are trivially SRGs) we have that each number $k$, $\lambda$ and $\mu$ appears at least once as an entry of $\bfA^2$, proving they are nonnegative integers.

In this work, we only need two basic facts about SRGs beyond~\eqref{equation.definition of srg} itself.
The first is a well-known relationship between their four parameters:
since $\bfA\bfone=k\bfone$, conjugating~\eqref{equation.definition of srg} by $\bfone$ gives
\begin{equation*}
k^2v
=\bfone^*\bfA^2\bfone
=(\lambda-\mu)\bfone^*\bfA\bfone+(k-\mu)\bfone^*\bfI\bfone+\mu\bfone^*\bfJ\bfone
=(\lambda-\mu)kv+(k-\mu)v+\mu v^2,
\end{equation*}
at which point dividing by $v$ gives
\begin{equation}
\label{equation.srg parameter relation}
k(k-\lambda-1)
=(v-k-1)\mu.
\end{equation}
The second fact we shall need is that the \textit{complement} of an SRG---the graph obtained from it by disconnecting neighbors and connecting nonneighbors---is another SRG.
Indeed, the adjacency matrix of this complement is $\tilde{\bfA}=\bfJ-\bfI-\bfA$; substituting $\bfA=\bfJ-\bfI-\tilde{\bfA}$ into $\bfA\bfone=k\bfone$ gives that the complement is regular with $\tilde{k}=v-k-1$; substituting this same equation into~\eqref{equation.definition of srg} and then solving for $\tilde{\bfA}^2$ gives that the complement is strongly regular with parameters
\begin{equation}
\label{equation.SRG complement parameters}
(\tilde{v},\ \tilde{k},\ \tilde{\lambda},\ \tilde{\mu})
=(v,\ v-k-1,\ v-2k-2+\mu,\ v-2k+\lambda).
\end{equation}

As noted above, the graphs arising from real ETFs according to~\eqref{equation.ETF to SRG conversion} need not be regular, let alone strongly regular.
The traditional method for overcoming this is to sign a given real ETF $\set{\bfphi_i}_{i=1}^{n}$ in order to assume without loss of generality that $\ip{\bfphi_1}{\bfphi_i}>0$ for all $i=2,\dotsc,n$.
This ensures that our adjacency matrix $\bfA$ given in~\eqref{equation.ETF to SRG conversion} is of the form
\begin{equation}
\label{equation.n-1 adjacency}
\bfA=\left[\begin{array}{ll}0&\bfone^*\\\bfone&\bfA_0\end{array}\right]
\end{equation}
where $\bfA_0$ is an $(n-1)\times (n-1)$ adjacency matrix.
As detailed in~\cite{FickusW15}, adjacency matrices $\bfA$ of this form satisfy our real ETF condition~\eqref{equation.ETF condition on adjacency matrix} if and only if $\bfA_0$ is the adjacency matrix of an $\SRG(n-1,k,\lambda,\mu)$ where $k=2\mu$:
\begin{theorem}
\label{theorem.traditional ETF SRG equivalence}
Let $\set{\bfphi_i}_{i=1}^{n}$ be an ETF for $\bbR^m$ where $m<n$ and let $\alpha=\tfrac{n}{m}$ and $\beta=[\frac{n-m}{m(n-1)}]^{\frac12}$.
Assume without loss of generality that $\ip{\bfphi_1}{\bfphi_i}>0$ for all $i$.
Then letting $\bfA_0$ satisfy
\begin{equation*}
\left[\begin{array}{ll}0&\bfone^*\\\bfone&\bfA_0\end{array}\right]
=\bfA
=\tfrac1{2\beta}\bfPhi^*\bfPhi-\tfrac{\beta+1}{2\beta}\bfI+\tfrac12\bfJ
\end{equation*}
we have that $\bfA_0$ is the adjacency matrix of an $\SRG(v,k,\lambda,\mu)$ with
\begin{equation}
\label{equation.traditional ETF SRG equivalence 1}
v=n-1,
\qquad
k=\tfrac n2-1+\tfrac{\alpha-2}{2\beta},
\qquad
\mu=\tfrac k2.
\end{equation}

Conversely, let $\bfA_0$ be the adjacency matrix of an $\SRG(v,k,\lambda,\mu)$ with $\mu=\frac k2$.
Then
\begin{equation}
\label{equation.traditional ETF SRG equivalence 2}
m=\tfrac{v+1}{2}\biggset{1+\tfrac{v-2k-1}{[(v-2k-1)^2+4v]^{\frac12}}}
\end{equation}
is the unique choice of $m$ for which there exists $\beta>0$ such that
\begin{equation*}
\bfG
=\left[\begin{array}{ll}1&\beta\bfone^*\\\beta\bfone&2\beta\bfA_0+(\beta+1)\bfI-\beta\bfJ\end{array}\right]
\end{equation*}
is the Gram matrix of an ETF for $\bbR^m$ of $n=v+1$ vectors.
Here, $\beta$ is necessarily \smash{$[\frac{n-m}{m(n-1)}]^{\frac12}$}.

These two transformations are inverses.
In fact, for any real scalars $m$, $n$, $v$, $k$ where $v>0$ and $n>\max\set{m,1}$,
having $v$ and $k$ be given by~\eqref{equation.traditional ETF SRG equivalence 1} is equivalent to having $m$ be given by~\eqref{equation.traditional ETF SRG equivalence 2} while $n=v+1$.
\end{theorem}

We offer three remarks on the above result.
First, given an $m\times n$ real ETF, the integrality of the parameter $k$ in~\eqref{equation.traditional ETF SRG equivalence 1} is by no means obvious, and is in fact closely related to strong, previously known integrality conditions on the existence of real ETFs~\cite{HolmesP04,SustikTDH07,Waldron09}.
Second, we do not specify $\lambda$ in~\eqref{equation.traditional ETF SRG equivalence 1} since it is superfluous:
for any $\SRG(v,k,\lambda,\mu)$ where $k=2\mu$, solving for $\lambda$ in~\eqref{equation.srg parameter relation} quickly gives $\lambda=\tfrac{3k-v-1}{2}$.
Third, our formula for $k$ in~\eqref{equation.traditional ETF SRG equivalence 1} is slightly different from that given in~\cite{Waldron09}.
This is because~\eqref{equation.ETF to SRG conversion} converts Gram matrix values of $\beta$ and $-\beta$ to $1$ and $0$, respectively, whereas in~\cite{Waldron09} they are converted to $0$ and $1$, respectively.
We make this change because if we view the frame vectors as points on a sphere, it is geometrically more natural to identify points as neighbors when the angle between them is acute, as opposed to obtuse.
Regardless, this preference is of little consequence: applying one approach to a real ETF equates to applying the other to its Naimark complement, and the resulting SRGs are simply graph complements of each other.

\section{Equiangular tight frames with centroidal symmetry}

In the previous section, we discussed how real ETFs are traditionally signed so as to assume without loss of generality that their adjacency matrix $\bfA$, as defined in~\eqref{equation.ETF to SRG conversion}, is of the form~\eqref{equation.n-1 adjacency}.
When substituted into~\eqref{equation.ETF condition on adjacency matrix}, this ensures the corresponding induced subgraph on $n-1$ vertices is regular, and moreover, strongly regular.
This approach is clever, and leads to an equivalence between real ETFs and SRGs with $k=2\mu$.
Interestingly, it turns out that for some constructions of real ETFs, the entire $n$-vertex graph~\eqref{equation.ETF to SRG conversion} just happens to be regular.
This means no signing is necessary, and we can simply expand the square in~\eqref{equation.ETF condition on adjacency matrix} to write $\bfA^2$ as a linear combination of $\bfA$, $\bfI$ and $\bfJ$, thereby proving $\bfA$ gives an $n$-vertex SRG.
In this section, we explore what it means for $\bfA$ to be regular.
Much of this theory is generalizable to the complex setting, and we do so whenever possible.
In the next section, we then build on these results to establish a new equivalence between these ``regular" real ETFs and SRGs on $v=n$ vertices whose parameters satisfy $v=4k-2\mu-2\lambda$.

The following result gives a fundamental characterization of when an adjacency matrix arising from a real ETF happens to be regular:
\begin{theorem}
\label{theorem.characterizing centered and axial}
Let $\set{\bfphi_i}_{i=1}^{n}$ be an ETF for $\bbF^m$ and let $\alpha=\frac nm$ and $\beta=[\frac{n-m}{m(n-1)}]^{\frac12}$.
\begin{enumerate}
\alphi
\item
If $\bbF=\bbR$, then the $n$-vertex graph whose adjacency matrix is  $\bfA=\tfrac1{2\beta}\bfPhi^*\bfPhi-\tfrac{\beta+1}{2\beta}\bfI+\tfrac12\bfJ$ is regular if and only if $\bfone$ is an eigenvector for the Gram matrix $\bfPhi^*\bfPhi$.
\item
In general, $\bfone$ is an eigenvector for $\bfPhi^*\bfPhi$ if and only if either:
\begin{enumerate}
\romanii
\item
$\bfPhi^*\bfPhi\bfone=\bfzero$, which occurs precisely when $\bfone$ lies in the null space of $\bfPhi$, or
\item
$\bfPhi^*\bfPhi\bfone=\alpha\bfone$, which occurs precisely when $\bfone$ lies in the column space of $\bfPhi^*$.
\end{enumerate}
\end{enumerate}
\end{theorem}

\begin{proof}
For (a), note that since $\bfone$ is an eigenvector for both $\bfI$ and $\bfJ=\bfone\bfone^*$, then $\bfone$ is an eigenvector for $\bfA$ if and only if it is an eigenvector for $\bfPhi^*\bfPhi$.
For (b), note that since $(\bfPhi^*\bfPhi)^2=\alpha\bfPhi^*\bfPhi$, the only possible eigenvalues for $\bfPhi^*\bfPhi$ are $0$ and $\alpha$.
Thus, $\bfone$ is an eigenvector for $\bfPhi^*\bfPhi$ if and only if either $\bfPhi^*\bfPhi\bfone=\bfzero$ or $\bfPhi^*\bfPhi\bfone=\alpha\bfone$.
The first case is equivalent to having $\bfone\in\rmN(\bfPhi^*\bfPhi)=\rmN(\bfPhi)$, giving (i).
For (ii), note that if $\bfPhi^*\bfPhi\bfone=\alpha\bfone$, then $\bfone\in\rmC(\bfPhi^*\bfPhi)=\rmC(\bfPhi^*)$.
Conversely, if $\bfone\in\rmC(\bfPhi^*)$ then $\bfone=\bfPhi^*\bfy$ for some $\bfy\in\bbF^m$ and so $\bfPhi^*\bfPhi\bfone=\bfPhi^*\bfPhi\bfPhi^*\bfy=\bfPhi^*(\alpha\bfI)\bfy=\alpha\bfone$.
\end{proof}

To get a better understanding of these conditions, note $\bfPhi^*\bfPhi\bfone=\gamma\bfone$ if and only if
\begin{equation*}
\gamma
=(\bfPhi^*\bfPhi\bfone)(i)
=\sum_{j=1}^{n}(\bfPhi^*\bfPhi)(i,j)
=\sum_{j=1}^{n}\ip{\bfphi_i}{\bfphi_j}
=\biggip{\bfphi_i}{\sum_{j=1}^{n}\bfphi_j}
\end{equation*}
for all $i=1,\dotsc,n$.
Regarding $\set{\bfphi_i}_{i=1}^{n}$ as point masses on the unit sphere, this means that $\bfone$ is an eigenvector for $\bfPhi^*\bfPhi$ precisely when the angle between the $\bfphi_i$'s and their centroid $\frac{1}{n}\sum_{j=1}^n\bfphi_j$ is constant.
One way for this to happen---case (i) in the above theorem---is for their centroid to be zero, meaning the $\bfphi_i$'s are ``equally distributed" about the origin.
The only other way---case (ii) above---is for their centroid to form an axis, and for the $\bfphi_i$'s to be ``equally distributed" in an affine hyperplane that is orthogonal to this axis.
This perspective leads us to the following terminology:
\begin{definition}
\label{definition.centered and axial}
We say an ETF $\set{\bfphi_i}_{i=1}^{n}$ for $\bbF^m$ is \textit{centered} if $\bfone\in\rmN(\bfPhi)$, namely when $\bfPhi^*\bfPhi\bfone=\bfzero$.
We say it is \textit{axial} if $\bfone\in\rmC(\bfPhi^*)$, namely when $\bfPhi^*\bfPhi\bfone=\alpha\bfone$ where $\alpha=\tfrac nm$.
In either case, we say such an ETF has \textit{centroidal symmetry}.
\end{definition}

Note an ETF cannot simultaneously be both centered and axial since if $\bfone$ is an eigenvector for $\bfPhi^*\bfPhi$, its eigenvalue cannot be both $0$ and $\alpha$;
alternatively, $\bfone$ cannot simultaneously lie in $\rmN(\bfPhi)$ as well as its orthogonal complement $\rmN(\bfPhi)^\perp=\rmC(\bfPhi^*)$.
Moreover, many ETFs are not centered, and many are not axial;
to see this, it helps to first discuss some connections between these types of symmetry and the simplest methods for obtaining a new ETF from an existing one: rotating, permuting, signing and taking Naimark complements.

\begin{theorem}
\label{theorem.ETF operations}
Let $\set{\bfphi_i}_{i=1}^{n}$ be an ETF for $\bbF^m$.
\begin{enumerate}
\alphi
\item
For any unitary matrix $\bfU\in\bbF^{m\times m}$ and permutation $\sigma:[n]\rightarrow[n]$, the ETF $\set{\bfU\bfphi_{\sigma(i)}}_{i=1}^{n}$ is:
\begin{enumerate}
\romanii
\item
centered if and only if $\set{\bfphi_i}_{i=1}^{n}$ is centered;
\item
axial if and only if $\set{\bfphi_i}_{i=1}^{n}$ is axial.
\end{enumerate}\smallskip
\item
For any vector $\bfz\in\bbF^n$ with unimodular entries, the ETF $\set{z_i\bfphi_i}_{i=1}^{n}$ is
\begin{enumerate}
\romanii
\item
centered if and only if $\bfz\in\rmN(\bfPhi)$;
\item
axial if and only if $\bfz\in\rmC(\bfPhi^*)$.
\end{enumerate}\smallskip
\item
\begin{enumerate}
\romanii
\item
If $\set{\bfphi_i}_{i=1}^{n}$ is centered then all of its Naimark complements are axial.\\
Conversely, if $\set{\bfphi_i}_{i=1}^{n}$ has an axial Naimark complement then $\set{\bfphi_i}_{i=1}^{n}$ is centered.
\item
If $\set{\bfphi_i}_{i=1}^{n}$ is axial then all of its Naimark complements are centered.\\
Conversely, if $\set{\bfphi_i}_{i=1}^{n}$ has a centered Naimark complement then $\set{\bfphi_i}_{i=1}^{n}$ is axial.
\end{enumerate}
\end{enumerate}
\end{theorem}

\begin{proof}
Let $\bfPhi$ denote the synthesis operator of $\set{\bfphi_i}_{i=1}^{n}$.
\begin{enumerate}
\alphi
\item
The Gram matrix of $\set{\bfU\bfphi_{\sigma(i)}}_{i=1}^{n}$ is $(\bfU\bfPhi\bfPi)^*(\bfU\bfPhi\bfPi)=\bfPi^*\bfPhi^*\bfPhi\bfPi$ where $\bfPi$ is the corresponding permutation matrix.
Since \smash{$\bfPi^*=\bfPi^{-1}$} and $\bfPi\bfone=\bfone$, we see that $\bfone$ is an eigenvector for $\bfPi^*\bfPhi^*\bfPhi\bfPi$ with a given eigenvalue if and only if it is an eigenvector for $\bfPhi^*\bfPhi$ with this same eigenvalue.
\item
The synthesis operator of $\set{z_i\bfphi_i}_{i=1}^{n}$ is $\bfPhi\bfD$ where $\bfD$ is a diagonal matrix whose diagonal is $\bfz$.
As such, $\set{z_i\bfphi_i}_{i=1}^{n}$ is centered if and only if $\bfone\in\rmN(\bfPhi\bfD)$, namely when $\bfzero=\bfPhi\bfD\bfone=\bfPhi\bfz$, i.e.\ $\bfz\in\rmN(\bfPhi)$.
Similarly, $\set{z_i\bfphi_i}_{i=1}^{n}$ is axial if and only if $\bfone\in\rmC(\bfD^*\bfPhi^*)$, namely when there exists $\bfy\in\bbF^m$ such that $\bfone=\bfD^*\bfPhi^*\bfy$, namely when there exists $\bfy$ such that $\bfz=\bfD\bfone=\bfPhi^*\bfy$, which is equivalent to having $\bfz\in\rmC(\bfPhi^*)$.
\item
As noted in Lemma~\ref{lemma.Naimark complements}, if $\set{\tilde{\bfphi}_i}_{i=1}^{n}$ is any Naimark complement of $\set{\bfphi_i}_{i=1}^{n}$ then their Gram matrices satisfy
$\bfI=\tfrac1{\alpha}\bfPhi^*\bfPhi+\tfrac1{\tilde{\alpha}}\tilde{\bfPhi}^*\tilde{\bfPhi}$
where $\alpha=\frac nm$ and $\tilde{\alpha}=\frac{n}{n-m}$.
In particular, $\bfPhi^*\bfPhi\bfone=\bfzero$ if and only if $\tilde{\bfPhi}^*\tilde{\bfPhi}\bfone=\tilde{\alpha}\bfone$,
while $\bfPhi^*\bfPhi\bfone=\alpha\bfone$ if and only if $\tilde{\bfPhi}^*\tilde{\bfPhi}\bfone=\bfzero$.
\qedhere
\end{enumerate}
\end{proof}

From result (b) in particular, we see that if $\set{\bfphi_i}_{i=1}^{n}$ is any ETF, then the ETF $\set{z_i\bfphi_i}_{i=1}^{n}$ is not centered for all unimodular vectors $\bfz$ which lie outside of $\rmN(\bfPhi)$, and is not axial for all such $\bfz$ that lie outside of $\rmC(\bfPhi^*)$.
As such, we shouldn't expect a ``random" ETF to be centered or axial.
A more difficult problem is to determine which ETFs are sign-equivalent to one which is centered or axial.
That is, which ETF synthesis operators have the property that either their null spaces or row spaces contain a unimodular vector?

In the remainder of this section, we consider two well-known explicit constructions of ETFs, namely the harmonic ETFs of~\cite{StrohmerH03,XiaZG05,DingF07} and the Steiner ETFs of~\cite{FickusMT12,JasperMF14}, and discuss whether they can be centered or axial.

\subsection{Harmonic equiangular tight frames with centroidal symmetry}

Harmonic ETFs are formed by restricting the characters of a finite abelian group to difference sets~\cite{StrohmerH03,XiaZG05,DingF07}.
To be precise, let $\set{g_i}_{i=1}^{n}$ and $\set{\gamma_j}_{j=1}^{n}$ be enumerations of a finite abelian group $\calG$ and its Pontryagin dual $\Gamma\cong\calG$, respectively.
Let $\bfH$ be the corresponding character table, namely the $n\times n$ complex Hadamard matrix with entries $\bfH(i,j)=\gamma_j(g_i)$.
A \textit{harmonic frame} is obtained by extracting rows from $\bfH$.
Specifically, for any $m$-element subset $\calD=\set{g_{i_k}}_{k=1}^{m}$ of $\calG$, consider the $m\times n$ synthesis operator $\bfPhi$ given by $\bfPhi(k,j)=\frac1{\sqrt{m}}\bfH(i_k,j)$.
It is known that the columns of $\bfPhi$ are an ETF precisely when $\calD$ is a \textit{difference set} of $\calG$, namely when the cardinality of $\set{(d,d')\in\calD\times\calD: g=d-d'}$ is constant over all $g\neq 0$~\cite{XiaZG05,DingF07}.

For example, $\calD=\set{1,2,4}$ is a difference set in the cyclic group $\calG=\bbZ_7$ since, as seen in the following difference table, every nonzero element of this group can be written as a difference of elements from this subset in exactly one way:
\begin{equation*}
\begin{array}{c|ccc}
-&1&2&4\\
\hline
1&0&1&3\\
2&6&0&2\\
4&4&5&0
\end{array}\ .
\end{equation*}
This means we can form a $3\times 7$ ETF by extracting the three corresponding rows from a $7\times 7$ character table of $\bbZ_7$, which is a discrete Fourier transform;
letting $\omega=\exp(2\pi\rmi/7)$, we have
\begin{equation}
\label{equation.3x7 example}
\bfPhi
=\frac1{\sqrt{3}}\left[\begin{array}{lllllll}
1&\omega^{1}&\omega^{2}&\omega^{3}&\omega^{4}&\omega^{5}&\omega^{6}\\
1&\omega^{2}&\omega^{4}&\omega^{6}&\omega^{8}&\omega^{10}&\omega^{12}\\
1&\omega^{4}&\omega^{8}&\omega^{12}&\omega^{16}&\omega^{20}&\omega^{24}
\end{array}\right].
\end{equation}

Difference sets are a well-studied topic in combinatorial design; see~\cite{JungnickelPS07} for an overview.
Well-known examples include \textit{Singer} difference sets in cyclic groups which have \smash{$m=\frac{q^j-1}{q-1}$} and \smash{$n=\frac{q^{j+1}-1}{q-1}$} for any prime power $q$ and integer $j\geq2$, as well as \textit{McFarland} difference sets in certain noncyclic abelian groups which have \smash{$m=q^{j-1}\frac{q^j-1}{q-1}$} and \smash{$n=q^j(\frac{q^j-1}{q-1}+1)$} under the same restrictions on $q$ and $j$.
As we now demonstrate, all harmonic ETFs are either centered or axial:

\begin{theorem}
\label{theorem.difference set implies centroidal}
Letting $\calD$ be a difference set in a finite abelian group, the corresponding harmonic ETF is axial if $0\in\calD$ and is otherwise centered.
In particular, if there exists an $m$-element difference set in an abelian group of order $n>m$, then there exists $m\times n$ centered ETFs as well as $m\times n$ axial ETFs.
\end{theorem}

\begin{proof}
Let $\calD$ be a difference set in some finite abelian group $\calG$.
If $0\in\calD$, then the corresponding row of $\bfPhi$ is constant: if $g_{i_k}=0$ for some $k$ then for this $k$ and all $j=1,\dotsc,n$,
\begin{equation*}
\bfPhi(k,j)=\frac1{\sqrt{m}}\bfH(i_k,j)=\frac1{\sqrt{m}}\gamma_j(g_{i_k})=\frac1{\sqrt{m}}\gamma_j(0)=\frac1{\sqrt{m}}.
\end{equation*}
This means $\bfone\in\rmC(\bfPhi^*)$ and so Theorem~\ref{theorem.characterizing centered and axial} implies $\set{\bfphi_i}_{i=1}^{n}$ is axial.
If on the other hand $0\notin\calD$ then $\bfPhi$ only contains multiples of the nonconstant rows of $\bfH$.
Using the well-known fact that character tables are complex Hadamard matrices, this means the rows of $\bfPhi$ are orthogonal to the constant row of $\bfH$, namely $\bfone$.
That is, if $0\notin\calD$ then $\bfone\in\rmC(\bfPhi^*)^\perp=\rmN(\bfPhi)$, at which point Theorem~\ref{theorem.characterizing centered and axial} implies $\set{\bfphi_i}_{i=1}^{n}$ is centered.
The final conclusion follows from the fact that any translate $g+\calD$ of a difference set $\calD$ is another difference set: taking $m<n$ ensures $\calD$ is a proper subset of $\calG$, and so $g+\calD$ will contain $0$ when $g\in -D$ and will not contain $0$ otherwise.
\end{proof}

For example, the element $0$ does not lie in the difference set $\set{1,2,4}$ of $\bbZ_7$ and so the ETF given in~\eqref{equation.3x7 example} is centered.
Meanwhile, subtracting $1$ from this difference set yields another difference set $\set{0,1,3}$ and its corresponding $3\times 7$ ETF is axial.
We also remark that the complement of a difference set is another difference set.
For example, $\set{0,3,5,6}$ is also a difference set of $\bbZ_7$, and the corresponding $4\times 7$ harmonic ETF is a Naimark complement of~\eqref{equation.3x7 example}.
Note that since $0\in\set{0,3,5,6}$, this $4\times 7$ ETF is axial; this also follows from Theorem~\ref{theorem.ETF operations}.(c) and the fact that it has a centered Naimark complement~\eqref{equation.3x7 example}.

\subsection{Steiner equiangular tight frames with centroidal symmetry}

Steiner ETFs are formed by taking a tensor-like combination of a unimodular regular simplex with a special type of finite geometry~\cite{FickusMT12}.
Specifically, given integer parameters $2\leq k<v$, a corresponding \textit{balanced incomplete block design}, denoted a $\BIBD(v,k,1)$, is a $v$-element set $\calV$ along with a set $\calB$ of $k$-element subsets of $\calV$ (called \textit{blocks}) with the property that the number of blocks that contain a given vertex $v\in\calV$ is some constant $r$, as well as the property that any pair of distinct vertices $v,v'\in\calV$ is contained in exactly one block (a finite geometry version of the property that two distinct points determine a unique line).
Denoting the number of blocks as $b$, it is not hard to show that we necessarily have $r=\tfrac{v-1}{k-1}$ and $b=\tfrac{v(v-1)}{k(k-1)}$.
BIBDs of this type are also called $(2,k,v)$-\textit{Steiner systems}.
Arranging the indicator functions of our blocks as the rows of a $\set{0,1}$-valued incidence matrix $\bfB\in\bbR^{b\times v}$,
these properties are equivalent to having that each row of $\bfB$ sums to $k$, each column of $\bfB$ sums to $r$, and that the dot product of any two distinct columns of $\bfB$ is one.

For example, when $v=4$ and $k=2$ we can take our blocks to be all $2$-element subsets of $\calV=\set{1,2,3,4}$, namely $\calB=\set{\set{1,2},\set{3,4},\set{1,3},\set{2,4},\set{1,4},\set{2,3}}$.
Here, $r=3$, $b=6$, and the corresponding $6\times 4$ incidence matrix is
\begin{equation}
\label{equation.6x4 incidence}
\bfB
=\left[\begin{array}{cccc}1&1&0&0\\0&0&1&1\\1&0&1&0\\0&1&0&1\\1&0&0&1\\0&1&1&0\end{array}\right].
\end{equation}

There are several known infinite families of such BIBDs, including finite affine and projective geometries.
Moreover, as detailed in~\cite{FickusMT12}, every $\BIBD(v,k,1)$ yields an ETF of $n=v(r+1)$ vectors in a space of dimension $m=b$.
The construction involves replacing each ``$1$" in the incidence matrix $\bfB$ with a row of an $r\times (r+1)$ unimodular regular simplex $\bfS$,
which itself is obtained by removing any row from an $(r+1)\times(r+1)$ complex Hadamard matrix $\bfH$.
For example, for the $\BIBD(4,2,1)$ given in~\eqref{equation.6x4 incidence} we have $r=3$ and so we can use a $3\times 4$ unimodular regular simplex obtained by removing the top row of the standard $4\times 4$ Hadamard matrix:
\begin{equation}
\label{equation.3x4 unimodular simplex}
\bfH=\left[\begin{array}{rrrr}1&1&1&1\\1&-1&1&-1\\1&1&-1&-1\\1&-1&-1&1\end{array}\right],
\qquad
\bfS=\left[\begin{array}{rrrr}1&-1&1&-1\\1&1&-1&-1\\1&-1&-1&1\end{array}\right].
\end{equation}
In general, to form an ETF from a $\BIBD(v,k,1)$ and an $r\times(r+1)$ unimodular regular simplex $\bfS$ we convert each of the $v$ columns of the BIBD's $b\times v$ incidence matrix $\bfB$ into a $b\times(r+1)$ matrix by replacing each ``$1$" in that column with a distinct row of $\bfS$ and replacing each ``$0$" in that column with a $1\times(r+1)$ row of zeros.
Horizontally concatenating these $v$ matrices of size $b\times(r+1)$ yields a matrix of size $b\times v(r+1)$ whose columns, when normalized by a factor of \smash{$r^{-\frac12}$}, form an ETF~\cite{FickusMT12}.

For example, combining the $6\times 4$ BIBD incidence matrix $\bfB$ from~\eqref{equation.6x4 incidence} with the $3\times 4$ unimodular regular simplex $\bfS$ from~\eqref{equation.3x4 unimodular simplex} yields the following ETF of $16$ vectors in $\bbR^6$:
\begin{equation}
\label{equation.6x16 Steiner ETF}
\bfPhi
=\frac1{\sqrt{3}}\left[\begin{array}{rrrr|rrrr|rrrr|rrrr}
1&-1&\phantom{+{}}1&-1&\phantom{+{}}1&-1&\phantom{+{}}1&-1& 0& 0& 0& 0& 0& 0& 0& 0\\
0& 0& 0& 0& 0& 0& 0& 0&\phantom{+{}}1&-1&\phantom{+{}}1&-1&\phantom{+{}}1&-1&\phantom{+{}}1&-1\\
1&\phantom{+{}}1&-1&-1& 0& 0& 0& 0&\phantom{+{}}1&\phantom{+{}}1&-1&-1& 0& 0& 0& 0\\
0& 0& 0& 0&\phantom{+{}}1&\phantom{+{}}1&-1&-1& 0& 0& 0& 0&\phantom{+{}}1&\phantom{+{}}1&-1&-1\\
1&-1&-1&\phantom{+{}}1& 0& 0& 0& 0& 0& 0& 0& 0&\phantom{+{}}1&-1&-1&\phantom{+{}}1\\
0& 0& 0& 0&\phantom{+{}}1&-1&-1&\phantom{+{}}1&\phantom{+{}}1&-1&-1&\phantom{+{}}1& 0& 0& 0& 0
\end{array}\right].
\end{equation}
Here, the vertical lines delineate the submatrices of $\bfPhi$ arising from the individual columns of $\bfB$.

Note that in general, given any $\BIBD(v,k,1)$, a corresponding unimodular regular simplex $\bfS$ always exists.
For example, we can always form $\bfS$ by removing a row from an $(r+1)\times(r+1)$ discrete Fourier transform.
However, if we want the resulting ETF to be real we need to form $\bfS$ by removing a row from an $(r+1)\times(r+1)$ real Hadamard matrix $\bfH$.
Moreover, \textit{Fisher's inequality} on BIBD parameters states $b\geq v$ which in turn implies $r=\tfrac{bk}{v}\geq k\geq 2$.
As such, the known half of the Hadamard conjecture requires $r+1$ to be divisible by $4$, namely that $r\equiv 3\bmod 4$.

Note the ETF in~\eqref{equation.6x16 Steiner ETF} is centered since its columns sum to zero.
In fact, we have a stronger property: the columns of each of its $6\times 4$ block submatrices sum to zero.
This phenomena occurs in general:
\begin{theorem}
\label{theorem.BIBD implies centered}
If there exists a $\BIBD(v,k,1)$ then there exists a centered ETF of $v(r+1)$ vectors for $\bbF^b$ where \smash{$r=\tfrac{v-1}{k-1}$} and \smash{$b=\tfrac{v(v-1)}{k(k-1)}$}.
Moreover, if there exists a real Hadamard matrix of size $r+1$, this ETF can be chosen to be real.
\end{theorem}

\begin{proof}
Let $\bfH$ be any complex Hadamard matrix of size $r+1$, such as a discrete Fourier transform.
Dividing the columns of $\bfH$ by their first entries, we can assume without loss of generality that every entry of the first row of $\bfH$ is $1$.
Removing this first row yields an $r\times(r+1)$ unimodular regular simplex $\bfS$.
Since the rows of $\bfH$ are orthogonal, every row of $\bfS$ is orthogonal to $\bfone$, meaning the columns of $\bfS$ sum to zero.
As such, the columns of the corresponding Steiner ETF $\bfPhi$ sum to zero, meaning it is centered.
For the final conclusion, note that if $\bfH$ can be chosen to be a real Hadamard matrix, this same argument shows that $\bfPhi$ is a real centered ETF.
\end{proof}

As we shall see in the next section, not every $\BIBD(v,k,1)$ yields an axial ETF.
In particular, we can use results from the SRG literature to conclude that there is no axial $7\times 28$ real ETF.
This is despite the fact that there is a $7\times 28$ real Steiner ETF arising from a $\BIBD(7,3,1)$ known as the \textit{Fano plane}.
Nevertheless, as we now discuss, there are infinite families of BIBDs that do yield axial ETFs.

In particular, a \textit{parallel class} of a $\BIBD(v,k,1)$ is a subcollection of its blocks $\calB$ that forms a partition of its vertex set $\calV$.
For example, for the $\BIBD(4,2,1)$ whose incidence matrix is given in~\eqref{equation.6x4 incidence}, $\set{\set{1,2},\set{3,4}}$ is a parallel class,
and so are both $\set{\set{1,3},\set{2,4}}$ and $\set{\set{1,4},\set{2,3}}$.
When this happens---when the blocks themselves can be partitioned so that each subcollection of blocks is a partition of the vertex set---the BIBD is called \textit{resolvable}, and is commonly denoted an $\RBIBD(v,k,1)$.
As detailed in~\cite{JasperMF14}, there are several known infinite families of RBIBDs, and their resulting Steiner ETFs can be rotated so that the entries of their synthesis matrices have constant modulus, a property helpful in waveform design and coding applications.
As we now demonstrate, every BIBD that contains at least one parallel class---including every RBIBD---yields an axial Steiner ETF:

\begin{theorem}
\label{theorem.RBIBD implies axial}
If there exists a $\BIBD(v,k,1)$ that contains at least one parallel class,
then there exists an axial ETF of $v(r+1)$ vectors for $\bbF^b$ where \smash{$r=\tfrac{v-1}{k-1}$} and \smash{$b=\tfrac{v(v-1)}{k(k-1)}$}.
Moreover, if there exists a real Hadamard matrix of size $r+1$, this ETF can be chosen to be real.
\end{theorem}

\begin{proof}
Without loss of generality, the rows of the BIBD's $b\times v$ incidence matrix $\bfB$ can be arranged so that the top $\frac{v}{k}$ rows form a partition for $\calV$.
That is, the top $\frac{v}{k}$ rows of $\bfB$ sum to a row of ones.
Let $\bfH$ be any complex Hadamard matrix of size $r+1$.
Dividing the columns of $\bfH$ by their first entries, we can assume without loss of generality that every entry of the first row of $\bfH$ is $1$.
Now form a unimodular regular simplex $\bfS$ by removing any row of $\bfH$ but its first.
That is, let all the entries of the first row of $\bfS$ be $1$.
Altogether, these facts imply that the first $\frac{v}{k}$ rows of $\bfPhi$ sum to a row whose entries are constant.
That is, $\bfone\in\rmC(\bfPhi^*)$ and so $\bfPhi$ is axial.
Finally, note that if $\bfH$ can be chosen to be real, then the resulting $\bfPhi$ is real.
\end{proof}

For example, consider the Steiner ETF that arises from the $\RBIBD(4,2,1)$ given in~\eqref{equation.6x4 incidence} and the $3\times 4$ unimodular regular simplex $\bfS$ obtained by removing the second row of the real Hadamard matrix $\bfH$ in~\eqref{equation.3x4 unimodular simplex}, namely
\begin{equation*}
\bfPhi
=\frac1{\sqrt{3}}\left[\begin{array}{rrrr|rrrr|rrrr|rrrr}
1&\phantom{+{}}1&\phantom{+{}}1&\phantom{+{}}1&\phantom{+{}}1&\phantom{+{}}1&\phantom{+{}}1&\phantom{+{}}1& 0& 0& 0& 0& 0& 0& 0& 0\\
0& 0& 0& 0& 0& 0& 0& 0&\phantom{+{}}1&\phantom{+{}}1&\phantom{+{}}1&\phantom{+{}}1&\phantom{+{}}1&\phantom{+{}}1&\phantom{+{}}1&\phantom{+{}}1\\
1&\phantom{+{}}1&-1&-1& 0& 0& 0& 0&\phantom{+{}}1&\phantom{+{}}1&-1&-1& 0& 0& 0& 0\\
0& 0& 0& 0&\phantom{+{}}1&\phantom{+{}}1&-1&-1& 0& 0& 0& 0&\phantom{+{}}1&\phantom{+{}}1&-1&-1\\
1&-1&-1&\phantom{+{}}1& 0& 0& 0& 0& 0& 0& 0& 0&\phantom{+{}}1&-1&-1&\phantom{+{}}1\\
0& 0& 0& 0&\phantom{+{}}1&-1&-1&\phantom{+{}}1&\phantom{+{}}1&-1&-1&\phantom{+{}}1& 0& 0& 0& 0
\end{array}\right].
\end{equation*}
Clearly, the constant vectors are in the row space of this matrix, since they lie in the span of the first two rows.
As such, this ETF is axial.
In the next section, we discuss how real axial Steiner ETFs in particular lead to new examples of SRGs.

\section{Real equiangular tight frames with centroidal symmetry}

In Section~2, we reviewed how every real ETF can be signed so as to equate it with an SRG on $v=n-1$ vertices in which $\mu=\tfrac k2$.
In this section, we show that every real ETF with centroidal symmetry equates to an SRG on $v=n$ vertices where $v=4k-2\lambda-2\mu$.
When combined with the results of the previous section, this new equivalence will imply the existence of several new SRGs.
We begin with some basic facts about this type of SRG:
\begin{lemma}
\label{lemma.SRGs with v=4k-2 lambda-2 mu}
An $\SRG(v,k,\lambda,\mu)$ has $v=4k-2\lambda-2\mu$ if and only if $\mu=\tfrac{k}{2}\tfrac{v-2k-2}{v-2k-1}$.
In particular, in this case $v-2k-1$ divides $k$.
Moreover, the graph complement of an SRG of this type is another SRG of this type.
\end{lemma}

\begin{proof}
First assume $v=4k-2\lambda-2\mu$.
Substituting $\lambda=2k-\mu-\frac v2$ into~\eqref{equation.srg parameter relation} gives
\begin{equation*}
(v-k-1)\mu
=k(k-\lambda-1)
=k(k-2k+\mu+\tfrac v2-1)
=\tfrac k2(v-2k-2)+k\mu.
\end{equation*}
Solving for $\mu$ then gives $\mu=\frac{k}{2}\frac{v-2k-2}{v-2k-1}$;
note here that since $v=4k-2\lambda-2\mu$ we know $v$ is even and so $v-2k-1\neq0$.
Conversely, assume $\mu=\tfrac{k}{2}\tfrac{v-2k-2}{v-2k-1}$ and so~\eqref{equation.srg parameter relation} becomes
\begin{equation*}
k(k-\lambda-1)
=(v-k-1)\mu
=(v-k-1)\tfrac{k}{2}\tfrac{v-2k-2}{v-2k-1}.
\end{equation*}
Solving for $\lambda$ gives $\lambda=(k-1)-\tfrac{v-k-1}{2}\tfrac{v-2k-2}{v-2k-1}$ and so we see that
\begin{equation*}
4k-2\lambda-2\mu
=4k-2(k-1)+(v-k-1)\tfrac{v-2k-2}{v-2k-1}-k\tfrac{v-2k-2}{v-2k-1}
=2k+2+(v-2k-2)
=v,
\end{equation*}
as claimed.
For the second claim, note that since $v-2k-1$ is relatively prime to $v-2k-2$, the integrality of $\mu$ requires $v-2k-1$ to divide $k$.
For the final conclusion, note the parameters~\eqref{equation.SRG complement parameters} of the graph complement of an $\SRG(v,k,\lambda,\mu)$ with $v=4k-2\lambda-2\mu$ satisfy
\begin{equation*}
4\tilde{k}-2\tilde{\lambda}-2\tilde{\mu}
=4(v-k-1)-2(v-2k-2+\mu)-2(v-2k+\lambda)
=4k-2\lambda-2\mu
=v
=\tilde{v}.\qedhere
\end{equation*}
\end{proof}

In the following subsection, we show that any centroid-symmetric real ETF leads to an SRG of this type,
and discuss the ramifications of this fact.

\subsection{Constructing SRGs from real ETFs with centroidal symmetry}

In the following result, we show that any $m\times n$ real ETF with centroidal symmetry yields an SRG on $v=n$ vertices;
this is in contrast to the traditional theory in which an $m\times n$ real ETF is equated to an SRG on $v=n-1$ vertices, cf.~Theorem~\ref{theorem.traditional ETF SRG equivalence}.

\begin{theorem}
\label{theorem.real centroid-symmetric ETF implies SRG}
Let $\set{\bfphi_i}_{i=1}^{n}$ be a centered or axial ETF for $\bbR^m$ where $m<n$, cf.~Definition~\ref{definition.centered and axial}.
Letting $\alpha=\tfrac{n}{m}$ and $\beta=[\frac{n-m}{m(n-1)}]^{\frac12}$ we have $\bfA=\tfrac1{2\beta}\bfPhi^*\bfPhi-\tfrac{\beta+1}{2\beta}\bfI+\tfrac12\bfJ$ is the adjacency matrix of an $\SRG(v,k,\lambda,\mu)$ with $v=4k-2\lambda-2\mu$.
In particular, centered ETFs have parameters
\begin{equation*}
v=n,
\qquad
k=\tfrac{n-1}{2}-\tfrac{1}{2\beta},
\qquad
\lambda=\tfrac n4-1+\tfrac{\alpha-4}{4\beta},
\qquad
\mu=\tfrac{n}4-\tfrac{\alpha}{4\beta},
\end{equation*}
while for axial ETFs we have
\begin{equation*}
v=n,
\qquad
k=\tfrac{n-1}{2}+\tfrac{\alpha-1}{2\beta},
\qquad
\lambda=\tfrac n4-1+\tfrac{3\alpha-4}{4\beta},
\qquad
\mu=\tfrac{n}4+\tfrac{\alpha}{4\beta}.
\end{equation*}
Moreover, the parameters of the graph complement of the SRG obtained from a centered (axial) ETF equal the parameters of the SRG obtained from its axial (centered) Naimark complements.
\end{theorem}

\begin{proof}
Writing $\bfPhi^*\bfPhi$ in terms of $\bfA$ and then substituting this relation into $\alpha\bfPhi^*\bfPhi=(\bfPhi^*\bfPhi)^2$ gives
\begin{equation*}
2\alpha\beta\bfA+\alpha(\beta+1)\bfI-\alpha\beta\bfJ
=\alpha[2\beta\bfA+(\beta+1)\bfI-\beta\bfJ]
=[2\beta\bfA+(\beta+1)\bfI-\beta\bfJ]^2.
\end{equation*}
To continue simplifying, we expand the above square.
This is possible since $\set{\bfphi_i}_{i=1}^{n}$ has centroidal symmetry:
Theorem~\ref{theorem.characterizing centered and axial} and Definition~\ref{definition.centered and axial} gives $\bfA\bfone=k\bfone$ for some $k$,
which implies $\bfA\bfJ=\bfA\bfone\bfone^*=k\bfone\bfone^*=k\bfJ$ and so $\bfJ\bfA=\bfJ^*\bfA^*=(\bfA\bfJ)^*=(k\bfJ)^*=k\bfJ$.
Thus,
\begin{align*}
2\alpha\beta\bfA+\alpha(\beta+1)\bfI-\alpha\beta\bfJ
&=[2\beta\bfA+(\beta+1)\bfI-\beta\bfJ]^2\\
&=4\beta^2\bfA^2+(\beta+1)^2\bfI+\beta^2 n\bfJ+4\beta(\beta+1)\bfA-2\beta(\beta+1)\bfJ-4\beta^2 k\bfJ\\
&=4\beta^2\bfA^2+4\beta(\beta+1)\bfA+(\beta+1)^2\bfI+[\beta^2 n-2\beta(\beta+1)-4\beta^2 k]\bfJ.
\end{align*}
Solving for $\bfA^2$ we see that it satisfies an equation of the form $\bfA^2=x\bfA+y\bfI+z\bfJ$ and so $\bfA$ is the adjacency matrix of a SRG on $v=n$ vertices.
Specifically,
\begin{align}
\nonumber
\bfA^2
&=\tfrac1{4\beta^2}\bigset{[2\alpha\beta-4\beta(\beta+1)]\bfA+[\alpha(\beta+1)-(\beta+1)^2]\bfI+[4\beta^2 k+2\beta(\beta+1)-\beta^2 n-\alpha\beta]\bfJ}\\
\label{equation.proof of CS ETF gives SRG 1}
&=\tfrac{\alpha-2\beta-2}{2\beta}\bfA+\tfrac{(\alpha-\beta-1)(\beta+1)}{4\beta^2}\bfI+\tfrac{4\beta k+2\beta+2-\beta n-\alpha}{4\beta}\bfJ.
\end{align}
Here, recall from~\eqref{equation.definition of srg} that the coefficients of $\bfA$, $\bfI$ and $\bfJ$ are $\lambda-\mu$, $k-\mu$ and $\mu$, respectively, and so
\begin{equation*}
\lambda+\mu
=(\lambda-\mu)+2\mu
=\tfrac{\alpha-2\beta-2}{2\beta}+\tfrac{4\beta k+2\beta+2-\beta n-\alpha}{2\beta}
=2k-\tfrac n2.
\end{equation*}
Thus, $v=n=4k-2\lambda-2\mu$, as claimed.
To determine explicit expressions for $k$ and $\mu$ in terms of $m$ and $n$,
recall Definition~\ref{definition.centered and axial} gives $\bfPhi^*\bfPhi\bfone=\gamma\bfone$ where either $\gamma=0$ or $\gamma=\alpha$;
for succinct calculation, we write this as $\gamma=\frac12(\alpha\pm\alpha)$, where ``$+$" and ``$-$" correspond to the axial and centered cases, respectively.
As such,
\begin{equation*}
k\bfone
=\bfA\bfone
=(\tfrac1{2\beta}\bfPhi^*\bfPhi-\tfrac{\beta+1}{2\beta}\bfI+\tfrac12\bfJ)\bfone
=(\tfrac{\alpha\pm\alpha}{4\beta}-\tfrac{\beta+1}{2\beta}+\tfrac{n}{2})\bfone
=(\tfrac{n-1}{2}+\tfrac{\alpha\pm\alpha-2}{4\beta})\bfone,
\end{equation*}
implying \smash{$k=\tfrac{n-1}{2}+\tfrac{\alpha\pm\alpha-2}{4\beta}$}; choosing ``$+$" or ``$-$" gives the expressions for $k$ given in the statement of the result.
This in turns implies the expressions for $\mu$: the coefficient of $\bfJ$ in~\eqref{equation.proof of CS ETF gives SRG 1} is
\begin{equation*}
\mu
=\tfrac{4\beta k+2\beta+2-\beta n-\alpha}{4\beta}
=\tfrac{[2\beta n-2\beta+\alpha\pm\alpha-2]+2\beta+2-\beta n-\alpha}{4\beta}
=\tfrac{\beta n\pm\alpha}{4\beta}
=\tfrac{n}{4}\pm\tfrac{\alpha}{4\beta}.
\end{equation*}
Together, these facts allow one to quickly discover that $\lambda=\tfrac{4k-2\mu-v}{2}=\tfrac n4-1+\tfrac{(2\pm1)\alpha-4}{4\beta}$.

To help clarify the proof of the final statement, here we refer to SRGs whose parameters satisfy $v=4k-2\lambda-2\mu$ as \textit{special}.
By Lemma~\ref{lemma.SRGs with v=4k-2 lambda-2 mu}, the four parameters of any special SRG are determined by just $v$ and $k$.
In particular, to show two special SRGs have the same parameters, it suffices to prove they have identical values of $v$ and $k$.
Now fix any $m\times n$ real ETF that is centered; the argument is similar in the axial case.
We already know this ETF yields a special SRG with $v=n$ and $k=\tfrac{n-1}{2}-\tfrac{1}{2\beta}$.
Moreover, Lemma~\ref{lemma.SRGs with v=4k-2 lambda-2 mu} tells us its graph complement is special with parameters $\tilde{v}=v=n$ and
\begin{equation}
\label{equation.proof of CS ETF gives SRG 2}
\tilde{k}
=v-k-1
=n-(\tfrac{n-1}{2}-\tfrac{1}{2\beta})-1
=\tfrac{n-1}{2}+\tfrac{1}{2\beta}.
\end{equation}
Meanwhile, by Theorem~\ref{theorem.ETF operations}, the $(n-m)\times n$ real Naimark complements of this ETF are axial and so yield a special SRG whose ``$v$" parameter is $n$ and whose ``$k$" parameter is
\begin{equation}
\label{equation.proof of CS ETF gives SRG 3}
\tfrac{n-1}{2}+\tfrac{\tilde{\alpha}-1}{2\tilde{\beta}},
\end{equation}
where \smash{$\tilde{\alpha}=\tfrac{n}{n-m}$} and \smash{$\tilde{\beta}=[\frac{m}{(n-m)(n-1)}]^{\frac12}$}.
Since \smash{$\frac{\tilde{\beta}}{\beta}=\tilde{\alpha}-1$}, \eqref{equation.proof of CS ETF gives SRG 2} equals \eqref{equation.proof of CS ETF gives SRG 3} and so these two special SRGs have the same parameters, as claimed.
\end{proof}

Two remarks about this theorem.
First, our expressions for $\lambda$ and $\mu$ are superfluous, and are only included for the sake of completeness:
since $v=4k-2\lambda-2\mu$, Lemma~\ref{lemma.SRGs with v=4k-2 lambda-2 mu} gives \smash{$\mu=\frac k2\frac{v-2k-2}{v-2k-2}$}, at which point $\lambda=\frac12(4k-2\mu-v)$.
Second, we point out that this expression for $\mu$ is very close to the ``$\mu=\frac k2$" relationship encountered when identifying real ETFs with SRGs on $v=n-1$ vertices, cf.~Theorem~\ref{theorem.traditional ETF SRG equivalence}.

As an example of Theorem~\ref{theorem.real centroid-symmetric ETF implies SRG}, consider the $7\times 28$ real Steiner ETF that arises from a $\BIBD(7,3,1)$.
By Theorem~\ref{theorem.BIBD implies centered}, this ETF can be chosen to be centered.
Applying Theorem~\ref{theorem.real centroid-symmetric ETF implies SRG} to it gives the existence of an $\SRG(28,12,6,4)$.
Meanwhile, applying Theorem~\ref{theorem.real centroid-symmetric ETF implies SRG} to any one of its $21\times 28$ real Naimark complements gives the existence of an $\SRG(28,15,6,10)$; this is not surprising, since the graph arising from these Naimark complements is simply the graph complement of our $\SRG(28,12,6,4)$.

While illustrative, the above example is not very exciting, as the existence of these particular SRGs is already well-known~\cite{Brouwer07,Brouwer15}.
For a novel result, consider a slightly different problem: does there exist a $7\times 28$ real ETF that has axial symmetry?
(Equivalently, does there exist a centered $21\times 28$ real ETF?)
Note such an ETF cannot arise from Theorem~\ref{theorem.RBIBD implies axial} since $3$ does not divide $7$, implying any $\BIBD(7,3,1)$ cannot contain a parallel class.
But, is there some other way to construct such an ETF?
The answer is no: if such an ETF exists then Theorem~\ref{theorem.real centroid-symmetric ETF implies SRG} implies the existence of an $\SRG(28,18,12,10)$,
contradicting known necessary conditions on the existence of SRGs including the \textit{Krein bound} and the \textit{absolute bound}~\cite{Brouwer07,Brouwer15}.

The above discussion illustrates how the contrapositive of Theorem~\ref{theorem.real centroid-symmetric ETF implies SRG} allows the SRG literature to better inform the search for ETFs.
But, it is also useful when applied directly: by combining it with the constructions given in Section~3, we can construct SRGs from ETFs.

For example, it is well-known that for any $j\geq 2$ there exists McFarland difference sets of $m=2^{j-1}(2^{j}-1)$ elements in $\bbZ_n$ where $n=2^{2j}$~\cite{JungnickelPS07}.
By Theorem~\ref{theorem.difference set implies centroidal}, there thus exists $m\times n$ harmonic real ETFs that are centered, and other $m\times n$ harmonic real ETFs that are axial.
Applying Theorem~\ref{theorem.real centroid-symmetric ETF implies SRG} to these then gives SRGs with parameters
\begin{equation}
\label{equation.real McFarland}
v=2^{2j},
\quad
k=2^{j-1}(2^j\pm1),
\quad
\lambda=2^{j-1}(2^{j-1}\pm1),
\quad
\mu=2^{j-1}(2^{j-1}\pm1),
\end{equation}
where ``$+$" and ``$-$" correspond to the axial and centered cases, respectively.
By consulting the tables of~\cite{Brouwer07,Brouwer15}, we see that SRGs with these parameters are already very well known.

For another direct application of Theorem~\ref{theorem.real centroid-symmetric ETF implies SRG}, note that if there exists a $\BIBD(v,k,1)$ and a real Hadamard matrix of size $r+1$ where \smash{$r=\tfrac{v-1}{k-1}$} then applying this result to the centered real ETFs of Theorem~\ref{theorem.BIBD implies centered} yields an SRG with parameters
\begin{equation*}
v_{\mathrm{srg}}=v(r+1),
\qquad
k_{\mathrm{srg}}=\tfrac{(v-1)(r+1)}2,
\qquad
\lambda=\tfrac{(v+k-4)(r+1)}{4},
\qquad
\mu=\tfrac{(v-k)(r+1)}{4}.
\end{equation*}
When deriving the above parameters, it helps to observe that $\alpha=\tfrac{k(r+1)}{r}$ and $\beta=\tfrac1r$ for all Steiner ETFs~\cite{FickusMT12}.
We also have added ``srg" to the subscripts of the $v$ and $k$ parameters of the SRG here so as to not confuse them with the $v$ and $k$ parameters of the BIBD.
These SRGs are also not new.
In fact, this SRG construction is identical to one of Goethals and Seidel~\cite{GoethalsS70}.
This is not a coincidence:
though we have always known that the $m\times n$ Steiner ETFs of~\cite{FickusMT12} are a generalization of the SRGs of~\cite{GoethalsS70},
we only recently realized that these SRGs have $v=n$, not $v=n-1$;
it was our inquiry into this fact that originally led us to the concept of centroidal symmetry.
That is, from this perspective, Theorems~\ref{theorem.BIBD implies centered} and~\ref{theorem.real centroid-symmetric ETF implies SRG} are an abstraction of the construction of~\cite{GoethalsS70}.
As we now discuss, this abstraction is useful: though applying Theorem~\ref{theorem.real centroid-symmetric ETF implies SRG} to centered Steiner ETFs merely recreates the work of~\cite{GoethalsS70},
applying this same theorem to axial Steiner ETFs gives new SRGs:
\begin{corollary}
\label{corollary.SRGs from RBIBDs}
If there exists a $\BIBD(v,k,1)$ that contains a parallel class and if there exists a real Hadamard matrix of size $r+1$ where \smash{$r=\tfrac{v-1}{k-1}$}, then applying Theorem~\ref{theorem.real centroid-symmetric ETF implies SRG} to the axial real ETFs of Theorem~\ref{theorem.RBIBD implies axial} yields an SRG with parameters
\begin{equation*}
v_{\mathrm{srg}}=v(r+1),
\quad
k_{\mathrm{srg}}=\tfrac{(v+k-1)(r+1)}2,
\quad
\lambda=\tfrac{(v+3k-4)(r+1)}4,
\quad
\mu=\tfrac{(v+k)(r+1)}4.
\end{equation*}
\end{corollary}
As a means of gauging the significance of this corollary,
Table~\ref{table.SRGs from axial Steiner ETFs} lists the parameters of all $\SRG(v,k,\lambda,\mu)$ with $v\leq 1300$ that are constructed by it.
This arbitrary upper bound on $v$ is chosen to be consistent with~\cite{Brouwer15};
as an artifact of this bound, nearly all of the SRGs in this table arise from BIBDs with block size $k=2$, while two of these BIBDs have $k=3$, and one has $k=5$.
Two of these twelve SRGs seem to be new, namely an $\SRG(780,410,220,210)$ which arises from an $\RBIBD(39,3,1)$ known as a~\textit{Kirkman triple system}~\cite{RayW71},
and an $\SRG(540,294,168,150)$ which arises from a $\BIBD(45,5,1)$ that contains a parallel class,
such as the following explicit construction of $99$ subsets of $\set{1,\dotsc,45}$ given by Colin~Barker~\cite{Gordon15}:
\begin{equation*}
\begin{tiny}
\begin{array}{rrrrrrrrrrrrrrrrrrrrrrrrrrrrrrrrr}
 1& 1& 1& 1& 1& 1& 1& 1& 1& 1& 1& 2& 2& 2& 2& 2& 2& 2& 2& 2& 2& 3& 3& 3& 3& 3& 3& 3& 3& 3& 3& 4& 4\\
 2& 6& 7& 8& 9&10&13&14&16&17&22& 6& 7& 8& 9&10&14&15&17&18&23& 6& 7& 8& 9&10&11&15&18&19&24& 6& 7\\
 3&15&12&27&19&11&18&25&32&23&30&12&11&13&28&20&19&21&33&24&26&16&13&12&14&29&22&20&34&25&27&30&17\\
 4&29&20&33&21&24&26&34&38&28&40&25&30&16&34&22&27&35&39&29&36&23&21&26&17&35&31&28&40&30&37&31&24\\
 5&44&35&36&45&39&37&41&43&31&42&40&45&31&37&41&38&42&44&32&43&42&36&41&32&38&43&39&45&33&44&39&43\\\\
 4& 4& 4& 4& 4& 4& 4& 4& 5& 5& 5& 5& 5& 5& 5& 5& 5& 5& 6& 6& 6& 6& 6& 6& 7& 7& 7& 7& 7& 8& 8& 8& 8\\
 8& 9&10&11&12&19&20&25& 6& 7& 8& 9&10&12&13&16&20&21& 7&13&14&18&20&21&14&15&16&19&22&11&15&17&20\\
14&13&15&16&23&35&21&28&11&26&18&15&14&17&24&22&31&29& 8&17&24&22&27&33&18&25&28&23&34&21&19&29&24\\
22&27&18&29&32&36&26&38&19&32&25&23&28&30&33&27&37&39& 9&38&26&28&32&37&39&27&33&29&38&28&40&34&30\\
37&42&33&40&44&41&34&45&34&40&44&38&43&36&45&35&42&41&10&41&35&36&45&43&42&31&41&37&44&32&43&42&38\\\\
 8& 9& 9& 9& 9& 9&10&10&10&10&10&11&11&11&11&12&12&12&13&13&13&14&14&14&15&15&15&16&21&26&31&36&41\\
23&11&12&16&18&24&12&13&17&19&25&12&17&18&26&18&19&27&19&20&28&16&20&29&16&17&30&17&22&27&32&37&42\\
35&20&22&25&30&31&16&23&21&26&32&13&25&23&33&21&24&34&22&25&35&21&23&31&24&22&32&18&23&28&33&38&43\\
39&36&29&26&35&40&37&30&27&31&36&14&35&27&38&31&28&39&32&29&40&30&33&36&34&26&37&19&24&29&34&39&44\\
45&44&33&39&43&41&45&34&40&44&42&15&37&41&42&38&42&43&39&43&44&44&40&45&36&45&41&20&25&30&35&40&45
\end{array}
\end{tiny}\ .
\end{equation*}
Here, a parallel class is found by inspection: $\set{1,2,3,4,5}$, $\set{6,7,8,9,10}$, $\set{11,12,13,14,15}$, etc.
To be clear, the most important contribution here is not these two new SRGs themselves.
Rather, they simply serve to demonstrate the novelty of Corollary~\ref{corollary.SRGs from RBIBDs} in general.

\setlength{\arraycolsep}{10pt}
\begin{table}
\begin{equation*}
\begin{array}{llllllllll}
\multicolumn{4}{c}{\text{BIBD}}&\multicolumn{2}{c}{\text{ETF}}&\multicolumn{4}{c}{\text{SRG}}\\
\hline
\multicolumn{1}{|l}{v}&k&r&\multicolumn{1}{l|}{b}&m&\multicolumn{1}{l|}{n}&v&k&\lambda&\multicolumn{1}{l|}{\mu}\\
\hline
   4&   2&   3&   6&   6&  16&  16&  10&   6&   6\\
   8&   2&   7&  28&  28&  64&  64&  36&  20&  20\\
  12&   2&  11&  66&  66& 144& 144&  78&  42&  42\\
  16&   2&  15& 120& 120& 256& 256& 136&  72&  72\\
  20&   2&  19& 190& 190& 400& 400& 210& 110& 110\\
  24&   2&  23& 276& 276& 576& 576& 300& 156& 156\\
  28&   2&  27& 378& 378& 784& 784& 406& 210& 210\\
  32&   2&  31& 496& 496&1024&1024& 528& 272& 272\\
  36&   2&  35& 630& 630&1296&1296& 666& 342& 342\\
  15&   3&   7&  35&  35& 120& 120&  68&  40&  36\\
  39&   3&  19& 247& 247& 780& 780& 410& 220& 210\\
  45&   5&  11&  99&  99& 540& 540& 294& 168& 150
\end{array}
\end{equation*}
\caption{
The parameters of SRGs with at most $1300$ vertices that arise from axial real Steiner ETFs constructed from BIBDs that contain at least one parallel class.
These SRGs are constructed via the process summarized in Corollary~\ref{corollary.SRGs from RBIBDs}.
To be precise, the first four columns give the parameters of a $\BIBD(v,k,1)$ that contains at least one parallel class;
by Theorem~\ref{theorem.RBIBD implies axial}, each of these implies the existence of an axial real ETF whose dimensions are given by the second two columns;
by Theorem~\ref{theorem.real centroid-symmetric ETF implies SRG}, each of these in turn implies the existence of an SRG whose parameters are given in the final four columns.
Consulting~\cite{Brouwer15}, the SRGs whose parameters are given in the last two rows, namely $\SRG(780,410,220,210)$ and $\SRG(540,294,168,150)$, seem to be new.
Note this automatically implies the existence of their graph complements, namely $\SRG(780,369,168,180)$ and $\SRG(540,245,100,120)$.
More importantly, this for the first time gives a single construction that explains the existence of all $12$ of these SRGs.
}
\label{table.SRGs from axial Steiner ETFs}
\end{table}
\setlength{\arraycolsep}{2pt}

To apply Corollary~\ref{corollary.SRGs from RBIBDs} more generally,
note that in order for a $\BIBD(v,k,1)$ to exist and contain at least one parallel class we need both $r=\tfrac{v-1}{k-1}$ and $\tfrac vk$ to be integers.
Since $k$ and $k-1$ are relatively prime, the Chinese Remainder Theorem then implies that $v$ necessarily satisfies $v\equiv k\bmod k(k-1)$.
Though this necessary condition on $v$ is not sufficient in general, it is known to suffice in certain special cases of $k$, such as $k=2$ and $k=3$, as detailed below.
Moreover, this condition is always asymptotically sufficient.
In fact, for any $k\geq 2$, there exists a positive integer $v_0(k)$ such that for any $v\geq v_0(k)$ for which $v\equiv k\bmod k(k-1)$, there exists an $\RBIBD(v,k,1)$~\cite{RayW73}.
Note that in order to apply Corollary~\ref{corollary.SRGs from RBIBDs}, we also need for there to exist a real Hadamard matrix of size $r+1$;
writing $v=k(k-1)w+k$ for some $w\geq 1$, this means that $r+1=\frac{v-1}{k-1}+1=kw+2$ is necessarily divisible by $4$, which rules out all $k$ that are divisible by $4$.

For example, in the special case where $k=2$, in order to have a $\BIBD(v,2,1)$ that contains a parallel class we necessarily have that $v$ is even.
Moreover, when $v$ is even, such a BIBD exists.
In fact, when $v$ is even, the \textit{round-robin tournament schedule} is a famous example of an $\RBIBD(v,2,1)$.
In order to apply Corollary~\ref{corollary.SRGs from RBIBDs}, we also need that $r+1=v$ is divisible by $4$.
Specifically, writing $v=4u$ for some $u\geq1$ gives the following result:
if there exists a Hadamard matrix of size $4u$, then applying Corollary~\ref{corollary.SRGs from RBIBDs} to an $\RBIBD(4u,2,1)$ yields an SRG with parameters
\begin{equation*}
v=16u^2,
\quad
k=2u(4u+1),
\quad
\lambda=u(4u+2),
\quad
\mu=u(4u+2).
\end{equation*}
All the SRGs in Table~\ref{table.SRGs from axial Steiner ETFs} with $k=2$ are of this type.
These are not new: if there exists a real Hadamard matrix of size $4u$ then there exists a real symmetric Hadamard matrix with constant diagonal of size $16u^2$, which implies the existence of SRGs with these parameters~\cite{GoethalsS70,BrouwervL84}.

For a new family, consider Corollary~\ref{corollary.SRGs from RBIBDs} in the special case where $k=3$.
Here, we necessarily have that $v\equiv 3\bmod 6$.
Moreover, for any such $v$, there always exists an $\RBIBD(v,3,1)$ known as a Kirkman triple system~\cite{RayW71}.
Writing $v=6w+3$, we also need that there exists a real Hadamard matrix of size $3w+2$, meaning $w=4u-2$ for some $u\geq 1$.
Altogether, we have the following result: if there exists a real Hadamard matrix of size $4(3u-1)$, applying Corollary~\ref{corollary.SRGs from RBIBDs} to an $\RBIBD(24u-9,3,1)$ yields an SRG with parameters
\begin{equation*}
v=4(24u-9)(3u-1),
\quad
k=2(24u-7)(3u-1),
\quad
\lambda=(24u-4)(3u-1),
\quad
\mu=(24u-6)(3u-1).
\end{equation*}
Letting $u=1$ and $u=2$ gives the two instances of SRGs in Table~\ref{table.SRGs from axial Steiner ETFs} that stem from BIBDs with $k=3$.
Though the first of these may have already been known to exist, the second was not known, indicating this is a new family of SRGs.
Since we, for example, can take $u=\tfrac{2^{2j-1}+1}{3}$ for any $j\geq1$, we see that this new family of SRGs is infinite.

Yet another family arises from the finite geometry of affine lines in $\bbF_q^j$ for a given prime power $q$ and $j\geq 2$:
applying Corollary~\ref{corollary.SRGs from RBIBDs} to this example of an $\RBIBD(q^j,q,1)$, we see that if there exists a real Hadamard matrix of size $\tfrac{q^j-1}{q-1}+1$, then there exists an SRG with
\begin{equation*}
v=q^j(\tfrac{q^j-1}{q-1}+1),
\quad
k=\tfrac{q^j+q-1}2(\tfrac{q^j-1}{q-1}+1),
\quad
\lambda=\tfrac{q^j+3q-4}4(\tfrac{q^j-1}{q-1}+1),
\quad
\mu=\tfrac{q^j+q}4(\tfrac{q^j-1}{q-1}+1).
\end{equation*}
Here, the smallest example that does not have $q=2$ seems to be $q=5$, $j=3$ which yields an $\SRG(4000,2064,1088,1040)$.

Finally, for any prime power $q$ there exists an $\RBIBD(q^3+q^2+q+1,q+1,1)$~\cite{Lorimer74}.
As such, if there exists a Hadamard matrix of size $q^2+q+2$, there exists an SRG with parameters
\begin{align*}
v&=(q^3+q^2+q+1)(q^2+q+2),&
k&=\tfrac12(q^3+q^2+2q+1)(q^2+q+2),\\
\lambda&=\tfrac14(q^3+q^2+4q)(q^2+q+2),&
\mu&=\tfrac14(q^3+q^2+2q+2)(q^2+q+2).
\end{align*}
Letting $q=2$ gives an $\SRG(120,68,40,36)$ listed in Table~\ref{table.SRGs from axial Steiner ETFs}.
The next smallest example is when $q=5$ which yields an $\SRG(4992,2576,1360,1296)$.

We are aware of two other general families of RBIBDs, namely \textit{Denniston designs} and \textit{unitals}.
However, in those cases, the requisite real Hadamard matrices do not exist~\cite{FickusMT12}.
As such, in order to squeeze more SRGs out of Corollary~\ref{corollary.SRGs from RBIBDs}, one needs to investigate BIBDs that are not necessarily resolvable and yet contain at least one parallel class, such as the $\BIBD(45,5,1)$ we discussed earlier; we leave a more general investigation of this for future work.
For the remainder of this paper, we return to the general theory relating SRGs to real ETFs with centroidal symmetry.

\subsection{Constructing real ETFs with centroidal symmetry from SRGs}

We now provide a converse to Theorem~\ref{theorem.real centroid-symmetric ETF implies SRG}, namely a method for converting an SRG whose parameters satisfy $v=4k-2\lambda-2\mu$ into a real ETF with centroidal symmetry.
As we shall see, whether this ETF is centered or axial depends on the sign of the quantity $v-2k-1$.
This makes sense since a Naimark complement of a centroid-symmetric ETF has the opposite type of symmetry, cf.~Theorem~\ref{theorem.ETF operations};
in terms of their graphs, this corresponds to taking a graph complement, which simply changes the sign of this parameter:
\begin{equation*}
\tilde{v}-2\tilde{k}-1
=v-2(v-2k-1)-1
=-(v-2k-1).
\end{equation*}

\begin{theorem}
\label{theorem.SRG implies real centroid-symmetric ETF}
Let $\bfA$ be the adjacency matrix of an $\SRG(v,k,\lambda,\mu)$ with $0<k<v-1$.
There exists $\beta>0$ and a positive integer $m$ such that $\bfG=2\beta\bfA+(\beta+1)\bfI-\beta\bfJ$ is the Gram matrix of an ETF $\set{\bfphi_i}_{i=1}^{n}$ for $\bbR^m$ if and only if $v=4k-2\lambda-2\mu$.
Moreover, in this case $\beta$ and $m$ are unique:
if $v-2k-1>0$ then this ETF is centered and
\begin{equation*}
m=\tfrac{v(v-2k-1)^2}{(v-1)+(v-2k-1)^2},
\qquad
\alpha=\tfrac nm=\tfrac{v-1}{(v-2k-1)^2}+1,
\qquad
\beta=[\tfrac{n-m}{m(n-1)}]^{\frac12}=\tfrac1{v-2k-1};
\end{equation*}
if $v-2k-1<0$ then this ETF is axial and
\begin{equation*}
m=\tfrac{v(v-1)}{(v-1)+(v-2k-1)^2},
\qquad
\alpha=\tfrac nm=\tfrac{(v-2k-1)^2}{v-1}+1,
\qquad
\beta=[\tfrac{n-m}{m(n-1)}]^{\frac12}=-\tfrac{v-2k-1}{v-1}.
\end{equation*}
Moreover, this method of converting an SRG with $v=4k-2\lambda-2\mu$ to a centroid-symmetric ETF is the inverse of the method of converting such an ETF into such an SRG given in Theorem~\ref{theorem.real centroid-symmetric ETF implies SRG}.
\end{theorem}

\begin{proof}
For any $\beta>0$, the matrix $\bfG=2\beta\bfA+(\beta+1)\bfI-\beta\bfJ$ is self-adjoint, has ones along its diagonal, and has off-diagonal entries of $\pm\beta$.
As such, Lemma~\ref{lemma.Gram matrix ETF characterization} gives that $\bfG$ is the Gram matrix of an ETF precisely when there exists $\alpha\in\bbR$ such that $\bfG^2=\alpha\bfG$.
Here, since $\bfA$ is the adjacency matrix of an $\SRG(v,k,\lambda,\mu)$ we have $\bfA\bfone=k\bfone$ and so $\bfA\bfJ=k\bfJ=\bfJ\bfA$ and moreover that $\bfA^2$ satisfies~\eqref{equation.definition of srg}.
As such,
\begin{align*}
\bfG^2
&=[2\beta\bfA+(\beta+1)\bfI-\beta\bfJ]^2\\
&=4\beta^2\bfA^2+(\beta+1)^2\bfI+\beta^2 v\bfJ+4\beta(\beta+1)\bfA-2\beta(\beta+1)\bfJ-4\beta^2 k\bfJ\\
&=4\beta^2[(\lambda-\mu)\bfA+(k-\mu)\bfI+\mu\bfJ]+4\beta(\beta+1)\bfA+(\beta+1)^2\bfI+[\beta^2 v-2\beta(\beta+1)-4\beta^2 k]\bfJ\\
&=2\beta[\beta(2\lambda-2\mu+2)+2]\bfA+[4\beta^2(k-\mu)+(\beta+1)^2]\bfI-\beta[\beta(4k-4\mu-v+2)+2]\bfJ.
\end{align*}
Since $0<k<v-1$, $\bfA$ is neither the complete graph nor its complement, making $\bfA$ linearly independent from $\bfI$ and $\bfJ$.
Thus, there exists $\alpha\in\bbR$ such that $\bfG^2=\alpha\bfG=2\beta\alpha\bfA+(\beta+1)\alpha\bfI-\beta\alpha\bfJ$ if and only if their coefficients of $\bfA$, $\bfI$ and $\bfJ$ are equal, that is, if and only if $\alpha$ satisfies the following three equations:
\begin{align}
\label{equation.proof of SRG gives CS ETF 1}
\alpha&=\beta(2\lambda-2\mu+2)+2,\\
\label{equation.proof of SRG gives CS ETF 2}
(\beta+1)\alpha&=4\beta^2(k-\mu)+(\beta+1)^2,\\
\label{equation.proof of SRG gives CS ETF 3}
\alpha&=\beta(4k-4\mu-v+2)+2.
\end{align}

In particular, if $\bfG=2\beta\bfA+(\beta+1)\bfI-\beta\bfJ$ is the Gram matrix of some ETF for some $\beta>0$ then both~\eqref{equation.proof of SRG gives CS ETF 1} and~\eqref{equation.proof of SRG gives CS ETF 3} hold, implying $2\lambda-2\mu=4k-4\mu-v$ and so $v=4k-2\lambda-2\mu$.
Conversely, if $v=4k-2\lambda-2\mu$, we claim $\bfG$ is the Gram matrix of some ETF for some $\beta>0$.
To see this, note the quadratic equation
\begin{equation*}
(\beta+1)[\beta(4k-4\mu-v+2)+2]=4\beta^2(k-\mu)+(\beta+1)^2,
\end{equation*}
is equivalent to the equation
\begin{equation}
\label{equation.proof of SRG gives CS ETF 5}
(v-1)\beta^2-(4k-4\mu-v+2)\beta-1=0,
\end{equation}
which has a unique positive solution by the quadratic formula.
Taking $\alpha=\beta(4k-4\mu-v+2)+2$ where $\beta$ is this unique positive root, $\alpha$ and $\beta$ immediately satisfy~\eqref{equation.proof of SRG gives CS ETF 2} and~\eqref{equation.proof of SRG gives CS ETF 3}.
Moreover, combining~\eqref{equation.proof of SRG gives CS ETF 3} and our assumption that $v=4k-2\lambda-2\mu$ gives~\eqref{equation.proof of SRG gives CS ETF 1}.

Having proven the ``if and only if" statement of the result, we from this point forward assume $v=4k-2\lambda-2\mu$ and focus on deriving explicit formulas for $m$ and $\beta$.
Note that in order for our positive parameter $\beta$ to satisfy both~\eqref{equation.proof of SRG gives CS ETF 2} and~\eqref{equation.proof of SRG gives CS ETF 3} it must be the unique positive root of~\eqref{equation.proof of SRG gives CS ETF 5}.
To obtain a more explicit expression for $\beta$, we use Lemma~\ref{lemma.SRGs with v=4k-2 lambda-2 mu} to simplify the linear coefficient of this polynomial:
\begin{equation}
\label{equation.proof of SRG gives CS ETF 6}
4k-4\mu-v+2
=4k-2k\tfrac{v-2k-2}{v-2k-1}-v+2
=\tfrac{v-1}{v-2k-1}-(v-2k-1).
\end{equation}
That is,~\eqref{equation.proof of SRG gives CS ETF 5} is $(v-1)\beta^2-\bigbracket{\tfrac{v-1}{v-2k-1}-(v-2k-1)}\,\beta-1=0$, meaning its positive root $\beta$ is
\begin{align*}
\beta
&=\tfrac1{2(v-1)}\Bigset{\bigbracket{\tfrac{v-1}{v-2k-1}-(v-2k-1)}+\Bigset{\bigbracket{\tfrac{v-1}{v-2k-1}-(v-2k-1)}^2+4(v-1)}^{\frac12}}\\
&=\tfrac1{2(v-1)}\Bigset{\bigbracket{\tfrac{v-1}{v-2k-1}-(v-2k-1)}+\bigabs{\tfrac{v-1}{v-2k-1}+(v-2k-1)}}.
\end{align*}
Take care to note the absolute value above;
as we noted earlier, the value of $v-2k-1$ can be either positive or negative for a given SRG,
since this value is negated by taking graph complements.
As such, we consider these two cases in order to further simplify this expression for $\beta$:
\begin{equation*}
\beta=\left\{\begin{array}{rl}\frac1{v-2k-1},&\ v-2k-1>0,\smallskip\\-\frac{v-2k-1}{v-1},&\ v-2k-1<0.\end{array}\right.
\end{equation*}
Recalling $\alpha=\beta(4k-4\mu-v+2)+2$ and~\eqref{equation.proof of SRG gives CS ETF 6} then gives
\begin{equation*}
\alpha
=\beta[\tfrac{v-1}{v-2k-1}-(v-2k-1)]+2
=\left\{\begin{array}{rl}\tfrac{v-1}{(v-2k-1)^2}+1,&\ v-2k-1>0,\smallskip\\\tfrac{(v-2k-1)^2}{v-1}+1,&\ v-2k-1<0.\end{array}\right.
\end{equation*}
In particular, as we discussed after Lemma~\ref{lemma.Gram matrix ETF characterization}, $\bfG$ is the Gram matrix of an ETF for $\bbR^m$ where $m$ is the multiplicity of $\alpha$ as an eigenvalue of $\bfG$, namely
\begin{equation*}
m
=\tfrac{\Tr(\bfG)}{\alpha}
=\tfrac{n}{\alpha}
=\tfrac{v}{\alpha}
=\left\{\begin{array}{rl}\tfrac{v(v-2k-1)^2}{(v-1)+(v-2k-1)^2},&\ v-2k-1>0,\smallskip\\\tfrac{v(v-1)}{(v-1)+(v-2k-1)^2},&\ v-2k-1<0.\end{array}\right.
\end{equation*}
Note that in this case, $\bfG$ is the Gram matrix of an $m\times n$ ETF and so we also necessarily have $\alpha=\tfrac nm$ and \smash{$\beta=[\tfrac{n-m}{m(n-1)}]^{\frac12}$} for this value of $m$.

For the final conclusion, note that applying Theorem~\ref{theorem.real centroid-symmetric ETF implies SRG} to a centroid-symmetric ETF $\set{\bfphi_i}_{i=1}^{n}$ for $\bbR^m$ produces the adjacency matrix $\bfA=\tfrac1{2\beta_0}\bfPhi^*\bfPhi-\tfrac{\beta_0+1}{2\beta_0}\bfI+\tfrac12\bfJ$ of an SRG with $v=4k-2\lambda-2\mu$.
Here, \smash{$\beta_0=[\tfrac{n-m}{m(n-1)}]^{\frac12}$} is the Welch bound of this ETF.
Applying our method here to $\bfA$ then gives a unique $\beta$ and $m$ for which $\bfG=2\beta\bfA+(\beta+1)\bfI-\beta\bfJ$ is the Gram matrix of an ETF.
Since choosing $\beta=\beta_0$ gives $\bfG=\bfPhi^*\bfPhi$ which is the Gram matrix of an ETF, we know our method here recovers $\bfPhi^*\bfPhi$.
Conversely,
if $\bfA$ is the adjacency matrix of any SRG with $v=4k-2\lambda-2\mu$, we know our approach here produces a unique $\beta$ and $m$ for which $\bfG=2\beta\bfA+(\beta+1)\bfI-\beta\bfJ$ is the Gram matrix of an ETF for $\bbR^m$.
Since this $\beta$ is necessarily the Welch bound of this ETF, applying Theorem~\ref{theorem.real centroid-symmetric ETF implies SRG} to it recovers $\bfA$.
\end{proof}

Together, Theorems~\ref{theorem.real centroid-symmetric ETF implies SRG} and~\ref{theorem.SRG implies real centroid-symmetric ETF} establish an equivalence between $m\times n$ real ETFs with centroidal symmetry and SRGs with $n=v=4k-2\lambda-2\mu$.
Moreover, combining Theorem~\ref{theorem.SRG implies real centroid-symmetric ETF} with the previously known equivalence between real ETFs and SRGs on $v=n-1$ vertices with $k=2\mu$ gives the following result:
\begin{corollary}
\label{corollary.SRG on n implies SRG on n-1}
If there exists an $\SRG(v,k,\lambda,\mu)$ with $v=4k-2\lambda-2\mu$ then there exists an
\begin{equation*}
\SRG(v-1,\ k\tfrac{v-2k}{v-2k-1},\ \tfrac{3k-v}{2}+\tfrac{3k}{2(v-2k-1)},\ \tfrac{k}2\tfrac{v-2k}{v-2k-1}).
\end{equation*}
\end{corollary}

\begin{proof}
If $v-2k-1>0$ then applying Theorem~\ref{theorem.traditional ETF SRG equivalence} to the centered ETF produced by Theorem~\ref{theorem.SRG implies real centroid-symmetric ETF} gives an $\SRG(v_0,k_0,\lambda_0,\mu_0)$ where $v_0=n-1=v-1$ and $\mu_0=\tfrac{k_0}2$ where
\begin{equation*}
k_0
=\tfrac n2-1+\tfrac{\alpha-2}{2\beta}
=\tfrac v2-1+\tfrac12\bigbracket{\tfrac{v-1}{(v-2k-1)^2}-1}(v-2k-1)
=k\tfrac{v-2k}{v-2k-1}.
\end{equation*}
If on the other hand $v-2k-1<0$ then applying Theorem~\ref{theorem.traditional ETF SRG equivalence} to the axial ETF produced by Theorem~\ref{theorem.SRG implies real centroid-symmetric ETF} gives an $\SRG(v_0,k_0,\lambda_0,\mu_0)$ where $v_0=n-1=v-1$ and $\mu_0=\tfrac{k_0}2$ where
\begin{equation*}
k_0
=\tfrac n2-1+\tfrac{\alpha-2}{2\beta}
=\tfrac v2-1-\tfrac12\bigbracket{\tfrac{(v-2k-1)^2}{v-1}-1}\tfrac{v-1}{v-2k-1}
=k\tfrac{v-2k}{v-2k-1}.
\end{equation*}
That is, in either case we get the same value of $v_0$ and $k_0$ and moreover that $\mu_0=\tfrac{k_0}2$ which in turn implies $\lambda_0=\tfrac{3k_0-v_0-1}{2}=\tfrac{3k-v}{2}+\tfrac{3k}{2(v-2k-1)}$.
\end{proof}

Looking at~\cite{Brouwer15}, this result appears to be new.
For example, it is known~\cite{Makhnev02,BannaiMV05} that there does not exist an $\SRG(1127,640,396,320)$ or an $\SRG(1127,486,165,243)$.
By the previously known equivalence between real ETFs and SRGs with $\mu=\tfrac k2$ given in Theorem~\ref{theorem.traditional ETF SRG equivalence},
this equivalently means that neither $47\times 1128$ nor $1081\times 1128$ real ETFs exist.
Of course, this means that neither axial $47\times 1128$ nor centered $1081\times 1128$ real ETFs exist.
By Theorems~\ref{theorem.real centroid-symmetric ETF implies SRG} and~\ref{theorem.SRG implies real centroid-symmetric ETF}, this equivalently means that $\SRG(1128,644,400,324)$ and $\SRG(1128,483,162,240)$ do not exist, a fact which can be directly proven using the Krein bound and the absolute bound~\cite{Brouwer15}.
But this same fact also means that neither centered $47\times 1128$ nor axial $1081\times 1128$ real ETFs exist.
Equivalently, we have that $\SRG(1128,560,316,240)$ and $\SRG(1128,567,246,324)$ do not exist, an observation that does not seem to have been made before~\cite{Brouwer15}.
This argument is summarized in Corollary~\ref{corollary.SRG on n implies SRG on n-1}:
if an $\SRG(1128,560,316,240)$ did exist then this corollary tells us there exists an $\SRG(1127,640,396,320)$, a contradiction of the results of~\cite{Makhnev02,BannaiMV05}.

Along these lines, subsequent to the initial posting of this work to arXiv.org,
the existence of SRGs with parameters $(75,32,10,16)$ have been ruled out using computer-assisted arguments~\cite{AzarijaM15}.
As detailed in~\cite{Yu15}, this means $19\times 76$ real ETFs do not exist,
and moreover Corollary~\ref{corollary.SRG on n implies SRG on n-1} implies SRGs with parameters $(76,30,8,14)$, $(76,45,28,24)$, $(76,35,18,14)$ and $(76,40,18,24)$ do not exist.
Even more recently, similar techniques have been used to show that SRGs with parameters $(95,40,12,20)$ do not exist~\cite{AzarijaM16}, thereby disproving the existence of $20\times 96$ real ETFs as well as SRGs with parameters $(96,38,10,18)$, $(96,57,36,30)$, $(96,45,24,18)$ and $(96,50,22,30)$.	

Applying other combinations of Theorems~\ref{theorem.traditional ETF SRG equivalence}, \ref{theorem.real centroid-symmetric ETF implies SRG} and~\ref{theorem.SRG implies real centroid-symmetric ETF} to the tables of~\cite{Brouwer07,Brouwer15} give other valuable insights into the nature of certain ETFs.
For example, since both $\SRG(16,6,2,2)$ and $\SRG(16,10,6,6)$ exist we know that there exist some $6\times 16$ real ETFs that are centered and that there exists others of this same size that are axial;
this also follows from applying~Theorem~\ref{theorem.difference set implies centroidal} to a $6$-element difference set in $\bbZ_2^4$, and is an instance of~\eqref{equation.real McFarland}.
For another example, note that though the problem of the existence of a $20\times 96$ real ETF remains open, the nonexistence of $\SRG(96,57,36,30)$ tells us that any such ETF cannot be axial.
Meanwhile, the existence of $\SRG(276,135,78,54)$ coupled with the nonexistence of an $\SRG(276,165,108,84)$ tells us that centered $27\times 276$ real ETFs exist but axial ones of this same size do not.
This is similar to the $7\times 28$ case we discussed earlier.
In other cases, the SRG tables tell us that real centered ETFs of a given size exist but that the existence of axial real ETFs of that same size remains open,
such as when $m=77$ and $n=210$.
In still others, such as the case of $19\times 76$ real ETFs, the tables tell us nothing.

Note these observations have a common theme: though the question of the existence of centered $m\times n$ real ETFs is equivalent to that of axial $(n-m)\times n$ real ETFs via Naimark complements---see Theorem~\ref{theorem.ETF operations}---it seems to be independent of the question of the existence of axial $m\times n$ real ETFs.
When both centered and axial real ETFs of the same size exist, such as when $m=6$ and $n=16$, we refer to the corresponding SRGs as \textit{centroidal complements}.
Note not every SRG with $v=4k-2\lambda-2\mu$ has a centroidal complement.
For example, $\SRG(28,12,6,4)$ exist and correspond to centered $7\times 28$ real ETFs whereas $\SRG(28,18,12,10)$, which correspond to axial $7\times 28$ real ETFs, do not.
As such, the following result, which allows one to explicitly compute the necessary parameters of a potential centroidal complement, is the strongest fact we could derive about such complements at this time:

\begin{theorem}
\label{theorem.centroidal complements}
$\SRG(v,k,\lambda,\mu)$ and $\SRG(\hat{v},\hat{k},\hat{\lambda},\hat{\mu})$ with $v=4k-2\lambda-2\mu$ and $\hat{v}=4\hat{k}-2\hat{\lambda}-2\hat{\mu}$ are centroidal complements if and only if
\begin{equation*}
-(v-2k-1)(\hat{v}-2\hat{k}-1)=v-1=\hat{v}-1.
\end{equation*}
\end{theorem}

\begin{proof}
First assume these two SRGs are centroidal complements,
namely that in accordance with Theorem~\ref{theorem.real centroid-symmetric ETF implies SRG} and without loss of generality,
$\SRG(v,k,\lambda,\mu)$ arises from an $n$-vector centered ETF in $\bbR^m$ for some $m$,
while $\SRG(\hat{v},\hat{k},\hat{\lambda},\hat{\mu})$ arises from an $n$-vector axial ETF for this same space.
We thus have $v=n=\hat{v}$, $k=\tfrac{n-1}{2}-\tfrac{1}{2\beta}$ and $\hat{k}=\tfrac{n-1}{2}+\tfrac{\alpha-1}{2\beta}$ where \smash{$\alpha=\tfrac nm$} and \smash{$\beta=[\tfrac{n-m}{m(n-1)}]^{\frac12}$}.
As such,
\begin{equation*}
-(v-2k-1)(\hat{v}-2\hat{k}-1)
=-[n-(n-1)+\tfrac{1}{\beta}-1][n-(n-1)-\tfrac{\alpha-1}{\beta}-1]
=\tfrac{\alpha-1}{\beta^2}
=n-1,
\end{equation*}
as claimed.
Conversely, suppose the parameters of these two SRGs satisfy $-(v-2k-1)(\hat{v}-2\hat{k}-1)=v-1=\hat{v}-1$.
As such, $v=\hat{v}$ and one of the quantities $v-2k-1$ and $\hat{v}-2\hat{k}-1$ is positive while the other is negative;
without loss of generality, $v-2k-1>0$.
Applying Theorem~\ref{theorem.SRG implies real centroid-symmetric ETF} to these SRGs then produces a centered ETF which spans a space of dimension  $m=\tfrac{v(v-2k-1)^2}{(v-1)+(v-2k-1)^2}$ as well as an axial ETF of this same dimension:
\begin{equation*}
\hat{m}
=\tfrac{\hat{v}(\hat{v}-1)}{(\hat{v}-1)+(\hat{v}-2\hat{k}-1)^2}
=\tfrac{v(v-1)}{(v-1)+\frac{(v-1)^2}{(v-2k-1)^2}}
=\tfrac{v(v-2k-1)^2}{(v-2k-1)^2+(v-1)}
=m.\qedhere
\end{equation*}
\end{proof}

Note that since taking a graph complement of an SRG with $v=4k-2\lambda-2\mu$ results in another graph of this type but negates the quantity $v-2k-1$,
Theorem~\ref{theorem.centroidal complements} immediately implies that if two SRGs are centroidal complements then their graph complements are as well.
More generally, Theorems~\ref{theorem.SRG implies real centroid-symmetric ETF} and~\ref{theorem.centroidal complements} suggest that SRGs with $v=4k-2\lambda-2\mu$ are probably best parameterized by $c=v-1$ and $d=v-2k-1$ as opposed to $v$ and $k$ themselves.
For example, centered $6\times16$ real ETFs are equivalent to $\SRG(16,6,2,2)$ which have $(c,d)=(15,3)$.
Negating $d$ corresponds to taking a graph complement;
in our example, we have $(\tilde{c},\tilde{d})=(c,-d)=(15,-3)$ which equates to $\SRG(16,9,4,6)$, i.e.\ axial $10\times 16$ real ETFs.
Meanwhile, negating the reciprocal of $d$ corresponds to taking a centroidal complement, e.g.\ $(\hat{c},\hat{d})=(c,-\tfrac{c}{d})=(15,-5)$ which equates to $\SRG(16,10,6,6)$, i.e.\ axial $6\times 16$ real ETFs.
These two operations commute and composing them gives, for example, \smash{$(\hat{\tilde{c}},\hat{\tilde{d}})=(\tilde{\hat{c}},\tilde{\hat{d}})=(c,\tfrac{c}{d})=(15,5)$} which equates to $\SRG(16,5,0,2)$,  i.e.\ centered $10\times 16$ real ETFs.

We conclude this paper with a few remarks on integrality conditions on the existence of real ETFs with centroidal symmetry.
One of the main results of~\cite{SustikTDH07} states that if an $m\times n$ real ETF exists with $1<m<n-1$ and $m\neq 2n$ then both \smash{$\tfrac1{\beta}=[\tfrac{m(n-1)}{n-m}]^{\frac12}$} and \smash{$\tfrac1{\tilde{\beta}}=[\tfrac{(n-m)(n-1)}{m}]^{\frac12}$} are necessarily odd integers.
We have not been able to strengthen these conditions into provably stricter ones on the existence of real ETFs with centroidal symmetry,
and leave a deeper investigation of this problem as future work.

We did however make some minor progress along these lines in the special case where $n=2m$.
There, \cite{SustikTDH07} gives that $m$ is necessarily odd and that $2m-1$ is necessarily the sum of two squares.
We claim that if an $m\times 2m$ real ETF has centroidal symmetry then $2m-1$ is necessarily a perfect square.
Indeed, for any such ETF, the Welch bound $(\frac1{2m-1})^{\frac12}$ and its negative occur a total number of $2m-1$ times in each row of the Gram matrix $\bfPhi^*\bfPhi$;
since the only eigenvalues of $\bfG$ are the rational numbers $0$ and $\alpha=\tfrac nm$, having $\bfone$ be an eigenvector of $\bfPhi^*\bfPhi$ thus forces $\beta$ to be rational as well.
This means, for example, that the famous $3\times 6$ real ETF constructed by identifying the antipodal points in the icosahedron is neither centered nor axial.
Meanwhile, combining Theorem~\ref{theorem.SRG implies real centroid-symmetric ETF} with~\cite{Brouwer07,Brouwer15} gives that there indeed exist centered $5\times 10$ real ETFs as well as axial $5\times 10$ real ETFs.

Remarkably, with the exception of real ETFs with $n=2m$, there seems to be no examples of real ETFs from~\cite{Brouwer15} in which we know there exists an $m\times n$ real ETF and also know that such an ETF can neither be centered nor axial.
This leads us to the following open problem: if there exists an $m\times n$ real ETF with $n\neq 2m$, then can we always either sign it to make it centered or sign it to make it axial?
That is, does a $\set{\pm1}$-valued vector always lie in the union of the null space and row space of the synthesis operator of a non-redundancy-two real ETF?
It is already known that non-equiangular two-distance tight frames have this property, cf.\ Theorem~2.4 of~\cite{BargGOY15}.
We leave a deeper investigation of this question for future work.

\section*{Acknowledgments}

This work was partially supported by NSF DMS 1321779, AFOSR F4FGA05076J002 and an AFOSR Young Investigator Research Program award.
The views expressed in this article are those of the authors and do not reflect the official policy or position of the United States Air Force, Department of Defense, or the U.S.~Government.

\end{document}